\documentclass{article}
 \usepackage{amsmath}
  \usepackage{amsthm}
 \usepackage{amssymb}

      \usepackage[english]{babel}
         \usepackage{wrapfig}
           \usepackage{graphicx}
             \usepackage{geometry}
               \usepackage[rightcaption]{sidecap}
                  \usepackage{tikz}
                   \usepackage{bbm}
						\usepackage{url}
\usepackage{amsfonts}
\usepackage{marvosym}
\usepackage{wasysym}
\usepackage{pifont}
\usepackage{mathrsfs}
\DeclareMathAlphabet{\mathpzc}{OT1}{pzc}{m}{it}

\allowdisplaybreaks[2]

\title{A Metric Lower Bound Estimate for Geodesics in the Space of K\"ahler Potentials}
\author{Jingchen Hu}

\newtheorem{problem}{Problem}[section]

\newtheorem{Proposition}{Proposition}[section]

\newtheorem{theorem}{Theorem}[section]

\newtheorem{remark}{Remark}[section]

\numberwithin{equation}{section}
\newtheorem{definition}{Definition}[section]
\newcommand{\ER}{\mathbb{R}}

\newcommand{\EC}{\mathbb{C}}

\newcommand{\MR}{\mathcal{R}} % Riemann Surface
\newcommand{\MV}{\mathcal{V}} % the Kahler manifold
\newcommand{\MH}{\mathcal{H}}
\newcommand{\MS}{\mathcal{S}}

\newcommand{\ddbar}{\partial\overline\partial}
\newcommand{\tautaubar}{\tau\overline{\tau}}

\newcommand{\taubar}{{\overline{\tau}}}

\newcommand{\abbar}{\alpha\overline{\beta}}
\newcommand{\covab}{{,\alpha\beta}}
\newcommand{\dzabbar}{dz^{\alpha}\wedge\overline{dz^\beta}}
\newcommand{\ijbar}{i\overline j}

\newcommand{\gammabar}{{\overline{\gamma}}}
\newcommand{\alphabar}{{\overline{\alpha}}}
\newcommand{\thetabar}{{\overline{\theta}}}
\newcommand{\betabar}{{\overline{\beta}}}
\newcommand{\mubar}{{\overline{\mu}}}
\newcommand{\zetabar}{{\overline{\zeta}}}
\newcommand{\psibar}{{\overline{\psi}}}
\newcommand{\etabar}{{\overline{\eta}}}
\newcommand{\lambdabar}{{\overline{\lambda}}}

\newcommand{\phibar}{{\overline{\phi}}}
\newcommand{\xibar}{{\overline{\xi}}}
\newcommand{\rhobar}{{\overline{\rho}}}

\newcommand{\jbar}{{\overline{j}}}
\newcommand{\ibar}{{\overline{i}}}

\newcommand{\bea}{\begin{align}}
\newcommand{\ena}{\end{align}}

\newcommand{\ConstA}{C_{An}}

\newcommand{\ConstB}{C_{Bn}}

\newcommand{\ConstBbar}{{C_{\overline B}}}

\newcommand{\ConstG}{C_{Gn}}

\newcommand{\CR}{C_R}

\newcommand{\TCB}{{\widetilde{C_B}}}

\newcommand{\TCA}{{\widetilde{C_A}}}

\newcommand{\TCG}{{\widetilde{C_G}}}
\newcommand{\HCG}{{\hat{C_G}}}
\newcommand{\TC}{{\widetilde{C}}}
\newcommand{\const}{\hat{C}}

\newcommand{\diameterR}{{d_\MR}}
\newcommand{\theoremdelta}{{\hat\delta}}

\newcommand{\dzbar}{{\overline{dz}}}
\newcommand{\zzbar}{{z\overline{z}}}
\newcommand{\zbar}{{\overline z}}
\newcommand{\bbar}{{\overline{b}}}

\newcommand{\X}{X}
\newcommand{\Xbar}{{\overline{X}}}
\newcommand{\cX}{{,X}}
\newcommand{\cXbar}{{,\overline{X}}}
\newcommand{\tr}{\text{tr}}

\newcommand{\Abar}{{\overline{A}}}
	\newcommand{\Bbar}{{\overline{B}}}
	\newcommand{\MP}{\mathcal{P}}
	\newcommand{\MU}{{\mathcal{U}}}
	\newcommand{\ML}{{\mathcal{L}}}
\begin{document}
\maketitle

\begin{abstract}
In this paper we establish a positive lower bound estimate for the second smallest eigenvalue of the complex Hessian of solutions to a degenerate complex Monge-Amp\`ere equation. As a consequence,  we find that in the space of K\"ahler potentials any two points close to each other in  $C^2$ norm can be connected by a  geodesic along which the associated metrics  do not degenerate.
\end{abstract}

\section{Introduction}\label{sec:intro}
%In section {}, we introduce the backgroud of our research and present
\subsection{Backgroud and Subject}\label{sec:sub}
In this paper we will be interested in the solution to degenerate complex Monge-Amp\`ere equations, which play important roles in K\"ahler geometry.

For a compact K\"ahler manifold $\MV$ with a metric $\omega_0$,  we define the following space of K\"ahler potentials, 
\begin{align}\MH=\{\varphi\in C^{\infty}(\MV)|\omega_0+\sqrt{-1}\ddbar\varphi>0\}.
\end{align}
Then a Riemannian metric can be introduced on this infinite dimensional space, for $\psi_1,\psi_2\in T_{\varphi}\MH$
\begin{align}\label{Hinnerproduct}<\psi_1, \psi_2>_\varphi=\int_\MV\psi_1\psi_2 \left(\omega_0+\sqrt{-1}\ddbar\varphi\right)^n.
\end{align}
Under this metric (\ref{Hinnerproduct}), a function $\Phi:[0,1]\times \MV\rightarrow \ER$, with $\Phi(t, \ast)\in \MH$ for each $t\in[0,1]$, is called a geodesic, if it satisfies
\begin{align}\label{realgeodesiceq}
\Phi_{tt}=\Phi_{t\betabar}g^{\alpha\betabar}\Phi_{\alpha t} , \ \ \ \ \ \ &\text{in\ \ \ \ } [0,1]\times \MV.
\end{align}
where $g_{\alpha\betabar}=\omega_{0,\alpha\betabar}+\Phi_{\alpha\betabar}$.
If we define $\mathcal{S}=\{\zeta\in\EC|  0 < \text{Re}(\zeta) < 1\}$, and consider a function $\Phi$ on $[0,1]\times\MV$ as a function on $\mathcal{S}\times \MV$, so that $\Phi(\zeta, \ast)$ does not depend on $\text{Im}(\zeta)$, then (\ref{realgeodesiceq}) is equivalent to the following homogenous complex Monge-Amp\`ere equation:
\begin{align}
\left(\Omega_0+\sqrt{-1}\ddbar\Phi\right)^{n+1} =0,\ \ \ \ \ \ &\text{in\ \ \ \ } \MS\times \MV.
\end{align}
Here if we denote the projection $\MR\times \MV\rightarrow \MV$ by $\pi_\MV$, then $\Omega_0=\pi_{{\MV}}^\ast \omega_0$.
This leads to the study of the following problem
\begin{problem}[Geodesic Problem in the Space of K\"ahler Potentials]
\label{ProblemGeodesic}
Given a K\"ahler manifold $(\MV, \omega_0)$ and two functions $\varphi_0, \varphi_1\in \MH$,  find a solution to the following Dirichlet boundary value problem on the space $\MS\times \MV$
\begin{align}
&\left(\Omega_0+\sqrt{-1}\ddbar\Phi\right)^{n+1}=0,& &\text{in\ \ \ \ } \MS\times \MV;\label{eq:geodesicwedgedeter}\\
&\ \ \ \ \ \ \ \ \Phi=\varphi_1,&  & \text{on\ \ }\{\text{Re}(\zeta)=1\}\times \MV;\\
&\ \ \ \ \ \ \ \ \Phi=\varphi_0,&  & \text{on\ \ }\{\text{Re}(\zeta)=0\}\times \MV;\\
&\Omega_0+\sqrt{-1}\ddbar\Phi \geq 0, &  &\text{in\ \ \ \ } \MS\times \MV,
\end{align}
and $\Phi$ is independent of $\text{Im}(\zeta).$
\end{problem}

\begin{remark}
	Solutions to Problem \ref{ProblemGeodesic} may not satisfy $ \Phi(\zeta, \ast)\in \MH$, for any $\zeta\in \MS$. So a solution is considered as a generalized(or weak) geodesic connecting $\varphi_0$ and $\varphi_1$.
\end{remark}
If we replace the infinite strip $\MS$ by a compact Riemann surface $\MR$, then Problem \ref{ProblemGeodesic} becomes: %\cite{{DonaldsonSymmetricSpace}}.
\begin{problem}[Dirichlet Problem for the Homogenous Monge-Amp\`ere Equation on a Product Space]
\label{ProblemHomoDisc}
Given a compact Riemann surface with boundary $\MR$, a K\"ahler manifold $(\MV, \omega_0)$ and a function $F\in C^{\infty}(\partial \MR\times \MV)$, which satisfies 
\begin{align}
\omega_0+\sqrt{-1}\ddbar F(\tau, \ast)>0, \ \ \text{for any $\tau \in \partial \MR$,}
\end{align}  find a solution to the following Dirichlet boundary value problem 
\begin{align}
\left(\Omega_0+\sqrt{-1}\ddbar\Phi\right)^{n+1}&=0, &\text{in\ \ \ \ } \MR\times \MV;\label{eq:wedgedeter}\\
\Phi&=F,& \text{on\ \ }\partial\MR\times \MV;\\
\Omega_0+\sqrt{-1}\ddbar\Phi &\geq 0, &\text{in\ \ \ \ } \MR\times \MV.
\end{align}
\end{problem}
\begin{remark}
Problem \ref{ProblemGeodesic} can be reduced to Problem \ref{ProblemHomoDisc} with $\MR=\{\tau| 1<|\tau|<2\}$, because there is a holomorphic covering map from $\MS$ to $\MR=\{\tau| 1<|\tau|<2\}$. 
\end{remark}

Problems \ref{ProblemGeodesic} and \ref{ProblemHomoDisc} are introduced in \cite{Semmes}\cite{SemmesGelfand} and \cite{DonaldsonSymmetricSpace}.
% The existence of weak solutions, which are maximal psh functions with $C^{1,1}$ interior regularity, can be derived by methods of \cite{BedfordTaylor} (see also section 4 of \cite{Klimek}).
 The global $C^{1,\overline{1}}$ regularity of the solution is established in \cite{Chen2000}. The result is later complemented by Blocki \cite{Blocki}, by showing that $\Phi$ is $C^{1,1}$ providing $(V, \omega_0)$ has non-negative bisectional curvature. In \cite{Berman} it's shown that, for any $\tau\in\MR$, $\Phi(\tau,\ast)$ has uniform $C^{1,1}$ bound, providing $\omega_0$ is integral. The complete $C^{1,1}$ regularity is established by \cite{Jianchun}\cite{JianchunFirst}. In \cite{LempertLizVivas} and \cite{LempertDarvas}, it is shown that $C^{1,1}$ is the optimal regularity for general solutions to Problem \ref{ProblemGeodesic}. 
 
 Even the optimal regularity for general solutions is known, it's still interesting to investigate if, in some situations, we can establish higher order regularity results.  Theorem 1 of \cite{DonaldsonHolomorphicDiscs} says that when the Riemann surface $\MR$ is a disc, the set of smooth functions $F$ for which a smooth solution to Problem \ref{ProblemHomoDisc} exists is open in  $C^{\infty}(\MR\times \MV)$. The method cannot be directly applied to other cases, where the Riemann surfaces are not discs. Similar ideas are used in \cite{Lempertmetrique} to construct smooth pluri-complex Green functions in a convex domain. In \cite{CFH}, by applying ideas of \cite{DonaldsonHolomorphicDiscs} and \cite{Lempertmetrique}, we show that for Problem \ref{ProblemGeodesic}, if the boundary values have small $C^5$ norm then the solution is $C^4$. In the opposite direction, in \cite{HuIMRN}, we construct a family of analytic functions $\varphi_k$, $k=1,2,...$, on $\MV$, and $|\varphi_k|_l\rightarrow 0$, for any $l>0$, but none of the $\varphi_k$ can be connected with $0$ by a smooth geodesic.

Another question for the degenerate Monge-Amp\`ere equation is whether the Hessian of a solution has maximal rank. In particular, for Problems \ref{ProblemGeodesic} and  \ref{ProblemHomoDisc}, we want to know, if for any $\tau\in \MS\ ( \text{or }\MR)$, 
\[\omega_0+\sqrt{-1}\ddbar\Phi(\tau,\ast)>0.\]
 This is shown to be wrong for general solutions. In \cite{RossNystrom}, when the Riemann surface is a disc, a solution $\Phi$ to Problem \ref{ProblemHomoDisc} is constructed, which satisfies
\[\Omega_0+\sqrt{-1}\ddbar\Phi=0,\]
in an open set in $\MR\times \MV$. However in this paper we will show a positive result when the boundary value $F$ has small $C^2$ norm. To present our result, we need to introduce the following approximation to Problem \ref{ProblemHomoDisc}:
\begin{problem}[Dirichlet Problem for the Non-Degenerate Monge-Amp\`ere Equation on a Product Space]
\label{ProblemDisc}
Given a bounded domain  with smooth boundary $\MR\subset \EC$, a K\"ahler manifold $(\MV, \omega_0)$ and a function $F\in C^{\infty}(\partial \MR\times \MV)$, which satisfies for any $\tau \in \partial \MR$
\begin{align}
\omega_0+\sqrt{-1}\ddbar F(\tau, \ast)>0, 
\end{align}find a solution to the following Dirichlet boundary value problem 
\begin{align}
\left(\Omega_0+\sqrt{-1}\ddbar\Phi\right)^{n+1}&=\epsilon \sqrt{-1}d\tau\wedge\overline { d \tau} \wedge\Omega_0^n, &\text{in\ \ \ \ } \MR\times \MV;\label{eq:wedgedeter}\\
\Phi&=F,& \text{on\ \ }\partial\MR\times \MV;\label{ProblemDiscBdCondidtion}\\
\Omega_0+\sqrt{-1}\ddbar\Phi &\geq 0, &\text{in\ \ \ \ } \MR\times \MV,
\end{align}
Here $\epsilon$ is a positive constant and $\tau$ is the complex coordinate on $\EC$.
\end{problem}
The solution to the problem above exists by the non-degenerate Monge-Amp\`ere equation theory, for example  \cite{GuanComplex} \cite{GT},  and a detailed discussion of this problem can be find in \cite{Blocki}. We know the problem has a smooth solution with
\[\Omega_0+\sqrt{-1}\ddbar\Phi > 0.\]
So for some positive constant $\sigma(\epsilon)$ depending on $\epsilon$,
\begin{align}\omega_0+\sqrt{-1}\ddbar\Phi(\tau, \ast)>\sigma(\epsilon){\omega_0}, \ \ \ \ \text{for all $\tau\in\MR$}.\label{quotientJune14}
\end{align}
% The $C^2$ estimate can be established by \cite{GuanComplex} or \cite{Chen2000}\cite{Jianchun}, then we use the weak Harnack estimates to prove the $C^{2,\alpha}$ estimates, $\alpha>0$,  \cite{GT}. Higher order estimates follow by using the bootstrap technique. However, the $C^{2,\alpha}$ estimates have to depend on $\epsilon.$
But the lower bound estimate (\ref{quotientJune14}), established using former elliptic theory, vanishes as $\epsilon\rightarrow 0$.
 In this paper, under some conditions, we want to establish a positive lower bound estimate for 
 \[\omega_0+\sqrt{-1}\ddbar\Phi(\tau, \ast)\]
which does not vanish as $\epsilon\rightarrow 0$. This leads to an estimate in the limiting case of $\epsilon=0,$ i.e. Problem \ref{ProblemDisc} and Problem \ref{ProblemGeodesic}.

%%%%%%%%%%%%%%%%%%%%%%%%%%%%%%%%%%%%
\subsection{Main Results of the Paper}\label{sec:result}
We will prove
\begin{theorem}[Metric Lower Bound Estimates for Monge-Amp\`ere Equations]\label{TheoremDisc}
Given a K\"ahler manifold $(\MV, \omega_0)$, a bounded domain with smooth boundary $\MR\subset \EC$ and a function $F\in C^{\infty}(\partial \MR\times \MV)$, there is a constant $\theoremdelta$ depending only on the dimension of $\MV$, the curvatures of $\MV$ and  their covariant derivatives, so that if 
\begin{align}0<\epsilon<\frac{\theoremdelta}{(\text{diameter of $\MR$})^2},
\end{align}and
\begin{align}
|F(\tau, \ast)|_{C^2}<\theoremdelta, \text{\ \ \ for all }\tau\in\partial\MR,
\end{align}
then the solution $\Phi$ to Problem \ref{ProblemDisc} satisfies
\[\omega_0+\sqrt{-1}\ddbar\Phi(\tau, \ast)
>\frac{1}{2} {\omega_0}, \text{\ \ \ \ \ \ \ \ \ for all \ \ \ }\tau\in\MR.\]
When measuring the diameter of $\MR$, we use the metric 
\[ds^2=\frac{1}{2}(d\tau\otimes \overline{d\tau}+\overline{d\tau}\otimes d\tau).\]
\end{theorem}

Theorem \ref {TheoremDisc} implies the following 
\begin{theorem}[Metric Lower Bound Estimate for Geodesics]\label{LowerBdGeodesic}
Given a K\"ahler manifold $(\MV, \omega_0)$, there is a constant $\delta$ depending on the dimension of $\MV$, the curvatures and the covariant derivatives of the curvatures of $\omega_0$, so that
for any $\varphi_0, \varphi_1\in \MH$, with 
\begin{align}
|\varphi_0|_{C^2}+|\varphi_1|_{C^{2}}\leq \delta, 
\end{align}measured by the metric $\omega_0$, the solution $\Phi$ to Problem \ref{ProblemGeodesic} satisfies
\[\omega_0+\sqrt{-1}\ddbar\Phi(\tau, \ast)
\geq\frac{1}{2} {\omega_0}, \text{\ \ \ \ \ \ \ \ \ for all \ \ \ }\tau\in\MS.\]
\end{theorem}
\subsection{Notation}\label{sec:notation}
 Roman indices, $i,j,k ,..., $ are from $0$ to $n$, where $n $ is the dimension of $\MV$. Greek indices, $\alpha, \beta, ...$, are from $1$ to $n$, except $\tau$ which is the coordinate on $\MR$; and  we consider the coordinate $\tau$ on the domain $\MR$ as the zeroth index. 

For the background metric $\omega_0$ on the K\"ahler manifold $\MV$, given local coordinates $\{z^\alpha\}$, we denote
\[\omega_0=\sqrt{-1} b_{\abbar}dz^\alpha \wedge \overline{dz ^\beta}.\] For a solution $\Phi$ to Problem \ref{ProblemDisc}, and for each $\tau \in \overline\MR$, denote
\[\omega_0+\sqrt{-1}\ddbar \Phi(\tau, \ast)=\sqrt{-1}g_{\abbar}\dzabbar.\]
We denote the determinant of $g_{\abbar}$ and $b_{\abbar}$ by $g$ and $b$, respectively. Then the equation 
(\ref{eq:wedgedeter}) takes another form
\begin{align}\label{eq:leafdeter}
\Phi_{\tautaubar}-\Phi_{\tau\overline\beta}g^{\abbar}\Phi_{\alpha\overline{\tau}}=\frac{\epsilon b}{g}. 
\end{align}

On $\MR\times \MV$, denote
\[\Omega_0+\sqrt{-1}\ddbar \Phi=\sqrt{-1}h_{\ijbar}dz^i\wedge \overline{dz^j}.\]
As a matrix 
\[(h_{\ijbar})=
\left(\begin{array}{cc}\Phi_{\tautaubar}& \Phi_{\tau\overline{\beta}}\\
											\Phi_{\alpha\overline{\tau}}& g_{\abbar}
       \end{array}
\right).
\]
When $\epsilon>0$, 
$(h_{\ijbar})$ is a positive definite Hermitian matrix, and $\Omega_0+\sqrt{-1}\ddbar\Phi$ defines a K\"ahler metric on the product space $\MR\times \MV$. The Laplace operator on the product space with this metric is 
\[\tilde L=h^{\ijbar}\partial_i\overline{\partial_j}.\]
In this paper, it's more convenient to use a scalar function multiple of this Laplace operator 
\[L=\frac{\epsilon b}{g}\tilde L.\]
We write 
\begin{align}
L=L^{\ijbar}\partial_i\overline{\partial_j},
\end{align}
where
\begin{align}\label{eq:operatorcoef}(L^{\ijbar})=
\left(\begin{array}{cc}1& -g^{\mu\overline{\beta}}\Phi_{\mu\overline\tau}\\
											-g^{\alpha\overline\theta}\Phi_{\tau\overline{\theta}}&g^{\alpha\overline\theta}\Phi_{\tau\overline{\theta}}g^{\mu\overline{\beta}}\Phi_{\mu\overline\tau}+ \frac{\epsilon b}{g}g^{\abbar}
       \end{array}
\right).
\end{align}
Here, $i$ is the column index and $j$ is the row index.
% We can see as $\epsilon\rightarrow 0$ the operator $L$ has uniformly bounded coefficients providing $\omega_0+\sqrt{-1}\ddbar\Phi(\tau, \ast)$ has a uniform lower bound.
It's also convenient to define $p^{\ijbar}$,
\begin{align}\label{pijbardef}(p^{\ijbar})=
\left(\begin{array}{cc}1& -g^{\mu\overline{\beta}}\Phi_{\mu\overline\tau}\\
											-g^{\alpha\overline\theta}\Phi_{\tau\overline{\theta}}&g^{\alpha\overline\theta}\Phi_{\tau\overline{\theta}}g^{\mu\overline{\beta}}\Phi_{\mu\overline\tau}
       \end{array}
\right).
\end{align}

Our curvature notation is 
\[-\nabla_\alpha\nabla_\gammabar\partial_\theta+\nabla_\gammabar\nabla_\alpha\partial_\theta
   ={R_{\alpha\gammabar\theta}}^\mu\partial_ \mu.\]
So, for a function $u$ on $\MV$, 
\begin{align}u_{,\alpha\beta\thetabar}=u_{,\alpha\thetabar\beta}-{R_{\beta\thetabar\alpha}}^\gamma u_\gamma,\label{thirdcommute}
	\end{align}
\begin{align}u_{, \alpha\beta\theta\gammabar}=u_{,\alpha\beta\gammabar\theta}
 																			-{R_{\theta\gammabar\beta}}^\mu u_{,\mu\alpha}
                      													-{R_{\theta\gammabar\alpha}}^\mu u_{,\mu\beta}.	\label{fourthcommute}
                      													\end{align}
\subsection{Structure of the Paper}\label{sec:structure}
 To prove Theorem \ref{TheoremDisc}, we need to control the second order derivatives of $\Phi$. So in section \ref{sec:subharm}, we construct a quantity
\begin{align}\label{eq:Qmention}
Q=\Phi_{\covab}\overline{\Phi_{,\gamma\theta}}g^{\alpha\overline{\theta}}g^{\beta\overline{\gamma}}+
		\Phi_{\abbar}\Phi_{\gamma\overline\theta}g^{\alpha\overline{\theta}}g^{\gamma\overline{\beta}}+
			\Phi_{\alpha}\overline{\Phi_{\beta}}g^{\alpha\overline \beta},
\end{align}
and show that when $\epsilon $ is small enough for a very small $\lambda\in \ER^+$, 
\[LQ>-\lambda Q.\]
The covariant derivatives in (\ref{eq:Qmention}) are taken with respect to the metric $\omega_0$.
To estimate $LQ$, we need to compute and estimate $L(\Phi_{\covab})$,  $L(\Phi_{\abbar})$ and $L(\Phi_{\alpha})$, these are done in section \ref{sec:a},  \ref{sec:b} and \ref{sec:gradient} respectively. How operator $L$ acts on tensors is explained in section \ref{sec:sec}.

In Appendix \ref{sec:app}, we carry out computations in the situation that $(\MV, \omega_0)$ is a 1-dimensional flat torus, in this case we get more complete results. Actually, the main idea of this paper is to generalize Theorem \ref{metriclowerbd} to general K\"ahler manifolds.

In Appendix \ref{sec:Leaf}, we study the limiting case of $\epsilon=0$. In this case, there is a leaf structure associated to a solution to the homogenous complex Monge-Amp\`ere equation, and the computation is simpler.

%%%%Section 1 ^^^ Section 1^^^Section 1 ^^^ Section 1^^^Section 1 ^^^ Section 1^^^Section 1 
%%%%%%%%%%%%%%%%%%%%%%%%%%%%%%%%%%%%%%%%%%%%%%%%%%
%%%%%============================================================================
%%%%%%%%%%%%%%%%%%%%%%%%%%%%%%%%%%%%%%%%%%%%%%%%%
%%%%Section2>>>Section2>>>Section2>>>Section2>>>Section2>>>Section2>>>Section2>>>
\section{Equations for $\Phi_{\alpha\betabar}$\ , $\Phi_{,\alpha\beta}$ and $\Phi_\theta$}\label{sec:sec}
%This section is a preparation for section \ref{sec:subharm}, where we computed $LQ$. 

In this section we compute and estimate $L(\Phi_{\covab})$,  $L(\Phi_{\abbar})$ and $L(\Phi_{\alpha})$. Here we consider $\Phi_{\abbar}$ as a section of the bundle $\pi_{\MV}^\ast(T^{\ast1,0}(\MV)\otimes T^{\ast 0,1}(\MV))$,  consider $\Phi_{\covab}$ as a section of the bundle $\pi_{\MV}^\ast(T^{\ast 1,0}(\MV)\otimes T^{\ast 1,0}(\MV))$ and consider $\Phi_\theta$ as a section of the bundle $\pi_{\MV}^\ast(T^{\ast 1,0}(\MV))$. These bundles have natural metrics induced by the metric $\omega_0$:
\begin{align}
\label{metricG}
<dz^\alpha, \overline{dz^{\beta}}>=&b^{\alpha\betabar},\\
\label{metricB}<dz^{\alpha}\otimes dz^\beta, \overline{dz^\theta\otimes dz^\gamma}>&=b^{\alpha\overline{\theta}}b^{\beta\overline{\gamma}},\\
\label{metricA}<dz^{\alpha}\otimes \overline{dz^\beta}, \overline{dz^\theta}\otimes dz^\gamma>&=b^{\alpha\overline{\theta}}b^{\gamma\overline{\beta}}.
\end{align}
We denote the norm of a tensor  by $|\cdot|$. 
For $\mathcal{F}=\mathcal{F}_{\alpha\betabar}dz^\alpha\otimes\overline{dz^{\beta}}$ and $\mathcal{G}=\mathcal{G}_{\alpha\beta}dz^\alpha\otimes{dz^{\beta}}$, 
\begin{align}
|\mathcal{F}|^2=\mathcal{F}_{\alpha\betabar}\overline{\mathcal{F}_{\theta\gammabar}}b^{\alpha\thetabar}b^{\gamma\betabar},
\label{ANormDef}\\
|\mathcal{G}|^2=\mathcal{G}_{\alpha\beta}\overline{\mathcal{G}_{\theta\gamma}}b^{\alpha\thetabar}b^{\beta\gammabar}.
\label{BNormDef}
\end{align}

The Chern connection on the bundle $T^{\ast 1,0}(\MV)$ naturally induces connections on 
$\pi_{\MV}^\ast(T^{\ast1,0}(\MV)\otimes T^{\ast 0,1}(\MV))$ and $\pi_{\MV}^\ast(T^{\ast 1,0}(V)\otimes T^{\ast 1,0}(\MV))$, and for these connections we have 
\begin{align}\label{connectionsA}
&\nabla_\tau( dz^\alpha\otimes dz^\betabar)=0,\ \ \ \nabla_\theta (dz^\alpha\otimes dz^\betabar)=(\nabla_\theta dz^\alpha)\otimes dz^{\betabar};\\ \label{connectionsB}
&\nabla_\taubar( dz^\alpha\otimes dz^\betabar)=0,\ \ \ \nabla_\thetabar (dz^\alpha\otimes dz^\betabar)= dz^\alpha\otimes \overline{(\nabla_\theta dz^{\beta})};\\ \label{barconnectionsB}
&\nabla_\tau (dz^\alpha\otimes dz^\beta)=0,\ \ \ \nabla_\theta( dz^\alpha\otimes dz^\beta)=(\nabla_\theta dz^\alpha)\otimes dz^{\beta}+ dz^{\alpha}\otimes (\nabla_\theta dz^\beta);\\
&\nabla_\taubar (dz^\alpha\otimes dz^\beta)=0,\ \ \ \nabla_\thetabar( dz^\alpha\otimes dz^\beta)=0.
\end{align}
The $\nabla_\tau, \nabla_\taubar$ derivatives are zero, because the metric $\omega_0$ is independent of the $\tau$ variable. The connections are also independent of $\tau$. As a consequence, $\nabla_\tau, \nabla_\taubar$ commute with all the $\nabla_\alpha, \ \nabla_{\betabar}$, so all components of the curvature tensor, containing $\tau, \ \taubar$, vanishes, i.e.
\begin{align}R_{\tau\alphabar\ast\ast}=R_{\taubar\alpha\ast\ast}=R_{\beta\taubar\ast\ast}=R_{\betabar\tau\ast\ast}=0.
\label{tauCurvature0}
\end{align}

For norms of curvature tensors  $R$, let
\begin{align}
|R|^2=R_{\alpha\etabar\beta\xibar}\overline{R_{\theta\zetabar\gamma\rhobar}}b^{\alpha\thetabar}b^{\beta\gammabar}b^{\zeta\etabar}b^{\rho\xibar}.
\label{CurvatureNormDef}
\end{align}For covariant derivatives of curvatures, the definition is analogous.

When acting on sections of bundles 
\[L=L^{\ijbar}\nabla_{\overline j}\nabla_i.\]

In addition, we denote
\begin{align}
B_{\alpha\beta}=\Phi_{\covab},\ \ \ \ \ \  A_{\alpha\overline\beta}=\Phi_{\abbar}.
\end{align}

%%%%%%%%%%%%%%%%%%%%%%%%%%%%%%%%%%%%%%%%%%%%%%%%
%%%%%%%%%%Section for LA>>>>>>%%%Section for LA>>>>>>%%%Section for LA>>>>>>%%%Section for LA>>>>>>%%%Section for LA>>>>>>%%%Section for LA>>>>>>%%%Section for LA>>>>>>%%%Section for LA>>>>>>%%%Section for LA>>>>>>%%%Section for LA>>>>>>%%%Section for LA>>>>>>%%%Section for LA>>>>>>
\subsection{Equations for $\Phi_{\abbar}$}\label{sec:a}
We apply $\partial_\theta$ to (\ref{eq:leafdeter}), and get
\begin{align}\label{equationthetader}
\Phi_{\tau\taubar\theta}-\Phi_{\tau\betabar\theta}g^{\alpha\betabar}\Phi_{\alpha\taubar}
											-\Phi_{\tau\overline{\beta}}g^{\alpha\betabar}\Phi_{,\alpha\taubar\theta}
								+\Phi_{\tau\betabar}g^{\alpha\mubar}\Phi_{,\zeta\mubar\theta}g^{\zeta\betabar}\Phi_{\alpha\taubar}=-\epsilon \frac{b}{g}g^{\alpha\betabar}\Phi_{,\alpha\betabar\theta}.
\end{align}
Here the covariant derivatives are taken with respect to the Chern connection of the metric (\ref{metricB}). 
Then we  apply $\partial_\gammabar$ to (\ref{equationthetader}) and get
\begin{align}
&\Phi_{\tau\taubar\theta\gammabar}
-\Phi_{,\tau\betabar\theta\gammabar}g^{\alpha\betabar}\Phi_{\alpha\taubar}-\Phi_{\tau\betabar\theta}g^{\alpha\betabar}\Phi_{\alpha\taubar\gammabar}+\Phi_{\tau\betabar\theta}g^{\alpha\mubar}\Phi_{,\phi\mubar\gammabar} g^{\phi\betabar}\Phi_{\alpha\taubar}\\
&-\Phi_{,\tau\betabar\gammabar}g^{\alpha\betabar}\Phi_{,\alpha\taubar\theta}-\Phi_{\tau\betabar}g^{\alpha\betabar}\Phi_{,\alpha\taubar\theta\gammabar}+\Phi_{\tau\betabar}g^{\alpha\mubar}\Phi_{,\phi\mubar\gammabar} g^{\phi\betabar}\Phi_{,\alpha\taubar\theta}\\
&+\Phi_{,\tau\betabar\gammabar}g^{\alpha\mubar}\Phi_{,\zeta\mubar\theta}g^{\zeta\betabar}\Phi_{\alpha\taubar}-\Phi_{\tau\betabar}g^{\alpha\etabar}\Phi_{,\lambda\etabar\gammabar}g^{\lambda\mubar}\Phi_{,\zeta\mubar\theta}g^{\zeta\betabar}\Phi_{\alpha\taubar}+\Phi_{\tau\betabar}g^{\alpha\mubar}\Phi_{,\zeta\mubar\theta\gammabar}g^{\zeta\betabar}\Phi_{\alpha\taubar}\\
&-\Phi_{\tau\betabar}g^{\alpha\mubar}\Phi_{,\zeta\mubar\theta}g^{\zeta\psibar}\Phi_{,\eta\psibar\gammabar}g^{\eta\betabar}\Phi_{\alpha\taubar}+\Phi_{\tau\betabar}g^{\alpha\mubar}\Phi_{,\zeta\mubar\theta}g^{\zeta\betabar}\Phi_{\alpha\taubar\gammabar}\\
=&- \frac{\epsilon b}{g}
\left(
g^{\alpha\betabar}\Phi_{,\alpha\betabar\theta\gammabar}-g^{\zeta\etabar}\Phi_{,\zeta\etabar\gammabar}g^{\alpha\betabar}\Phi_{,\alpha\betabar\theta}-g^{\alpha\mubar}\Phi_{,\zeta\mubar\gammabar}g^{\zeta\betabar}\Phi_{,\alpha\betabar\theta}
\right).
\end{align} 
Then we commute indices. In all terms containing $\theta, \gammabar$, we properly move $\theta, \gammabar$ leftwards and get 
\begin{align}
	&\Phi_{\theta\gammabar\tau\taubar}
	-\Phi_{,\theta\gammabar\tau\betabar}g^{\alpha\betabar}\Phi_{\alpha\taubar}
	-\Phi_{\theta\betabar\tau}g^{\alpha\betabar}\Phi_{\alpha\gammabar\taubar}
	+\Phi_{\theta\betabar\tau}g^{\phi\betabar}\Phi_{,\phi\gammabar\mubar}g^{\alpha\mubar} \Phi_{\alpha\taubar}\label{0622-19}\\
	&-\Phi_{,\betabar\gammabar\tau}g^{\alpha\betabar}\Phi_{,\alpha\theta\taubar}
	-\Phi_{,\theta\gammabar\alpha\taubar}\Phi_{\tau\betabar}g^{\alpha\betabar}+{R_{\alpha\gammabar\theta}}^{\eta}\Phi_{\eta\taubar}g^{\alpha\betabar}\Phi_{\tau\betabar}\label{0622-20}\\
	&+\Phi_{,\alpha\theta\taubar}g^{\alpha\mubar}\Phi_{,\mubar\gammabar\phi}g^{\phi\betabar}\Phi_{\tau\betabar}
	-\Phi_{,\alpha\theta\taubar}g^{\alpha\mubar}{R_{,\phi\gammabar\mubar}}^{\etabar}\Phi_{\etabar} g^{\phi\betabar}\Phi_{\tau\betabar}\label{0622-21}\\
	&+\Phi_{,\betabar\gammabar\tau}g^{\zeta\betabar}\Phi_{,\zeta\theta\mubar}g^{\alpha\mubar}\Phi_{\alpha\taubar}+\Phi_{,\betabar\gammabar\tau}g^{\zeta\betabar}{R_{\zeta\mubar\theta}}^\eta\Phi_\eta g^{\alpha\mubar}\Phi_{\alpha\taubar}\label{0622-22}\\
	&-\Phi_{,\lambda\gammabar\etabar}g^{\lambda\mubar}\Phi_{,\theta\mubar\zeta}g^{\zeta\betabar}\Phi_{\tau\betabar}g^{\alpha\etabar}\Phi_{\alpha\taubar}\label{0622-23}\\
	&+\Phi_{\tau\betabar}g^{\alpha\mubar}\Phi_{,\theta\gammabar\zeta\mubar}g^{\zeta\betabar}\Phi_{\alpha\taubar}+\left(({R_{\theta\mubar\zeta}}^\eta\Phi_\eta)_{,\gammabar}+({R_{\gammabar\zeta\theta}}^\eta\Phi_\eta)_{,\mubar}\right)\Phi_{\tau\betabar}g^{\alpha\mubar}g^{\zeta\betabar}\Phi_{\alpha\taubar}\label{0622-24}\\
	&-\Phi_{\tau\betabar}g^{\alpha\mubar}\Phi_{,\zeta\theta\mubar}g^{\zeta\psibar}\Phi_{,\gammabar\psibar\eta}g^{\eta\betabar}\Phi_{\alpha\taubar}\label{0622-25}
	\\&
	-\Phi_{\tau\betabar}g^{\alpha\mubar}\left({R_{\theta\mubar\zeta}}^\psi\Phi_{\psi}\Phi_{,\gammabar\phibar\eta}+{\Phi_{,\zeta\theta\mubar}}{R_{\phibar\eta\gammabar}}^{\rhobar}\Phi_{\rhobar}
+{R_{\theta\mubar\zeta}}^\psi\Phi_{\psi}{R_{\phibar\eta\gammabar}}^{\rhobar}\Phi_{\rhobar}\right)g^{\eta\betabar}\Phi_{\alpha\taubar}g^{\zeta\phibar}\label{0622-26}
	\\
	&+\Phi_{\tau\betabar}g^{\alpha\mubar}\Phi_{,\theta\mubar\zeta}g^{\zeta\betabar}\Phi_{\alpha\gammabar\taubar}\label{0622-27}\\
	=&- \frac{\epsilon b}{g}
	\left(
	g^{\alpha\betabar}\Phi_{,\theta\gammabar\alpha\betabar}\right)- \frac{\epsilon b}{g}g^{\alpha\betabar}
	\left(({R_{\theta\betabar\alpha}}^\mu\Phi_{\mu})_{,\gammabar}+({R_{\gammabar\alpha\theta}}^\rho\Phi_\rho)_{,\betabar}\right)\label{0622-28}\\
	&+\frac{\epsilon b}{g}\left(g^{\zeta\etabar}\Phi_{,\zeta\etabar\gammabar}g^{\alpha\betabar}\Phi_{,\alpha\betabar\theta}\right)+
	\frac{\epsilon b}{g}\left(g^{\alpha\mubar}\Phi_{,\zeta\gammabar\mubar}g^{\zeta\betabar}\Phi_{,\theta\betabar\alpha}\right)\label{0622-29}\\
	&+\frac{\epsilon b}{g}\left(g^{\alpha\mubar}\Phi_{,\mubar\gammabar\zeta}g^{\zeta\betabar}\Phi_{,\theta\alpha\betabar}\right)-\frac{\epsilon b}{g}\left(g^{\alpha\mubar}\Phi_{,\mubar\gammabar\zeta}g^{\zeta\betabar}\Phi_{,\theta\alpha\betabar}\right).\label{0622-30}
\end{align} 
When commuting  indices, we used (\ref{thirdcommute})  (\ref{fourthcommute}) and (\ref{tauCurvature0}). The terms in (\ref{0622-30}) are ficticious, they add up to be zero.

We will simplify the equation above to (\ref{equationAfirst}). To do this, we need the following notation: $(\ast. \ast)_k$ stands for the $k-$th term in $(\ast.\ast)$, including the sign. For example
 \begin{align}\label{termnotation}
 	(\ref{0622-28})_1=- \frac{\epsilon b}{g}
 \left(
 g^{\alpha\betabar}\Phi_{,\theta\gammabar\alpha\betabar}\right),\ \ \  (\ref{0622-19})_2=	-\Phi_{,\theta\gammabar\tau\betabar}g^{\alpha\betabar}\Phi_{\alpha\taubar}. \end{align}
   With this notation
 \begin{align}
 (\ref{0622-19})_1+ (\ref{0622-19})_2+(\ref{0622-20})_2+(\ref{0622-24})_1- (\ref{0622-28})_1&=L^{i\jbar}A_{\theta\gammabar,i\jbar};\\
- (\ref{0622-19})_3- (\ref{0622-19})_4- (\ref{0622-23})- (\ref{0622-27})+ (\ref{0622-29})_2
 &=L^{i\jbar}A_{\theta\mubar, i}g^{\zeta\mubar}A_{\zeta\gammabar,\jbar};
\\ -(\ref{0622-20})_1- (\ref{0622-21})_1- (\ref{0622-22})_1- (\ref{0622-25})+ (\ref{0622-30})_1
 &=L^{i\jbar}B_{\theta\mu,\jbar}g^{\mu\zetabar}\overline{B_{\zeta\gamma,\ibar}}.
 \end{align}
For all terms containing curvature $R$, we combine them into $U_{\theta\gammabar}$, and denote
\begin{align}
	 (\ref{0622-29})_1+ (\ref{0622-30})_2=F_{\theta\gammabar.}
\end{align}We get 
\begin{align}
\label{equationAfirst}
L^{\ijbar}A_{\theta\gammabar, \ijbar}=A_{\theta\mubar, i} g^{\zeta\mubar} A_{\zeta\gammabar, \jbar}L^{\ijbar}
+B_{\theta\mu,\jbar}g^{\mu\zetabar}\overline{B_{\zeta\gamma}}_{,i} L^{\ijbar}+F_{\theta\gammabar}+U_{\theta\gammabar}.
\end{align}
The expressions for $F, U$ are
\begin{align}\label{expressF}
F_{\theta\gammabar}=\frac{\epsilon b}{g}
\left(g^{\zeta\etabar}\Phi_{,\zeta\etabar\theta} g^{\alpha\betabar}\Phi_{,\alpha\betabar\gammabar}
						-\Phi_{,\theta\alpha\betabar}g^{\alpha\mubar}\Phi_{,\mubar\gammabar\zeta}g^{\zeta\betabar}\right),
\end{align}
and
\begin{align}
U_{\theta\gammabar}=&-{R_{\theta\gammabar\alpha}}^\eta\Phi_{\eta\taubar}g^{\alpha\betabar}\Phi_{\tau\betabar}
-({R_{\theta\mubar\zeta}}^\eta\Phi_\eta)_{,\gammabar} g^{\zeta\betabar}\Phi_{\tau\betabar}g^{\alpha\mubar}\Phi_{\alpha\taubar}
-({R_{\theta\gammabar\zeta}}^\eta\Phi_\eta)_{,\mubar} g^{\zeta\betabar}\Phi_{\tau\betabar}g^{\alpha\mubar}\Phi_{\alpha\taubar}\\
&-({R_{\theta\mubar\zeta}}^\lambda\Phi_\lambda )g^{\zeta\etabar}({R_{\phi\gammabar\etabar}}^\psibar \Phi_\psibar )g^{\phi\betabar}\Phi_{\tau\betabar}g^{\alpha\mubar}\Phi_{\alpha\taubar}\\
&+\frac{\epsilon b}{g}\left(({R_{\alpha\gammabar\theta}}^\eta\Phi_\eta)_{,\betabar}g^{\alpha\betabar}
   													-({R_{\alpha\betabar\theta}}^\eta\Phi_\eta)_{,\gammabar}g^{\alpha\betabar}\right)\\
&-({R_{\gammabar\mu\phibar}}^{\zetabar}\Phi_{\zetabar})g^{\mu\betabar}g^{\alpha\phibar}
\Phi_{,\theta\alpha\taubar}\Phi_{\tau\betabar}\label{A6_2}\\
&-({R_{\phi\gammabar\etabar}}^\psibar \Phi_{\psibar})g^{\zeta\etabar}\Phi_{,\theta\zeta\mubar}g^{\alpha\mubar}
    \Phi_{\alpha\taubar}g^{\phi\betabar}\Phi_{\tau\betabar}\label{A11_3}\\
&-({R_{\theta\mubar\zeta}}^\eta\Phi_\eta) g^{\zeta\betabar}\Phi_{,\betabar\gammabar\tau}g^{\alpha\mubar}
   \Phi_{\alpha\taubar}\label{A8_2}\\
&+({R_{\theta\mubar\zeta}}^\lambda \Phi_\lambda)g^{\zeta\etabar}\Phi_{,\etabar\gammabar\phi}
     g^{\alpha\mubar}\Phi_{\alpha\taubar}g^{\phi\betabar}\Phi_{\tau\betabar}.\label{A11_2}
\end{align}
\begin{remark}
$F$ vanishes when $\dim (\MV)=1$, and $U$ vanishes when the metric $\omega_0$ is flat (curvature $=0$), so when the K\"ahler manifold  $(\MV, \omega_0)$ is a  1-dimensional flat torus we have more complete results. This is discussed in Appendix \ref{sec:app}. 
\end{remark}
In section \ref{sec:subharm}, we will plug (\ref{equationAfirst}) into the expression of $LQ$ and try to get 
\[LQ>-\lambda Q.\]
So we need to use non-negative terms to control indefinite terms. The main non-negative terms from $LQ$ will be
\begin{align}\label{expressE}
E=B_{\theta\gamma,\alphabar} \overline{B_{\zeta\eta}}_{,\beta} g^{\beta\alphabar} g^{\theta\etabar} g^{\gamma\zetabar}
+A_{\theta\gammabar, \alpha}A_{\eta\zetabar,\betabar}g^{\alpha\betabar}g^{\eta\gammabar}g^{\theta\zetabar},
\end{align}
\begin{align}\label{expressP}
P=B_{\theta\gamma,\jbar} \overline{B_{\zeta\eta}}_{,i} p^{\ijbar} g^{\theta\etabar} g^{\gamma\zetabar}
+A_{\theta\gammabar, i}A_{\eta\zetabar,\jbar}p^{\ijbar}g^{\eta\gammabar}g^{\theta\zetabar},
\end{align}
and
\begin{align}\label{expressT}
T=\Phi_{\tau\betabar}\Phi_{\alpha\taubar}g^{\alpha\betabar}+\Phi_{\tau\alpha}\Phi_{\taubar\betabar}g^{\alpha\betabar}.
\end{align}
The $p^{\ijbar}$ in (\ref{expressP}) is defined by (\ref{pijbardef}).

In the previous expression of $U$, we need to combine (\ref{A6_2}) with (\ref{A11_3}) and combine (\ref{A8_2}) with (\ref{A11_2}), then we can control $U$ by $E, T, P$ and $Q$. By changing some dummy indices we get
\begin{align}(\ref{A6_2})+(\ref{A11_3})=(\Phi_{\tau\betabar}g^{\mu\betabar})({R_{\mu\gammabar\phibar}}^\xibar\Phi_{\xibar})g^{\alpha\phibar}(\Phi_{,\theta\alpha\taubar}-\Phi_{,\theta\alpha\etabar}g^{\lambda\etabar}\Phi_{\lambda\taubar}),\end{align}
\begin{align} (\ref{A8_2})+(\ref{A11_2})=-(\Phi_{\alpha\taubar}g^{\alpha\mubar})({R_{\theta\mubar\zeta}}^{\lambda}\Phi_{\lambda})g^{\zeta\betabar}(\Phi_{,\betabar\gammabar\tau}-\Phi_{\betabar\gammabar\phi}g^{\phi\rhobar}\Phi_{\tau\rhobar}).
	\end{align}
Then providing that $Q<1$ and
\begin{align}
{\frac{1}{2}}(b_{\alpha\betabar})<(g_{\alpha\betabar})<{\frac{3}{2}}(b_{\alpha\betabar} ) 																																						\label{metricassumption}
\end{align}we can find a constant $\ConstA$ which only depends on the dimension of $\MV$, so that
 \begin{align}\label{controlF}
|F|\leq \ConstA\frac{\epsilon b}{g}E,
\end{align}
 \begin{align}\label{controlU}
|U|\leq \ConstA (\max|R|+\max|R|^2+\max|\nabla R|)
 			\left(T+\sqrt{Q}\sqrt{T}\sqrt{P}+\frac{\epsilon b}{g}\sqrt{Q}\right).
\end{align}

\subsection{Equations for $\Phi_{\covab}$}\label{sec:b}
The computation for $B=\Phi_{\covab}dz^{\alpha}\otimes dz^{\beta}$ is similar to the previous computation for $A$. Apply $\partial_\theta$ and $\nabla_{\gamma}$ to equation (\ref{eq:leafdeter}) gives that
\begin{align}\label{equationB}
L^{\ijbar}B_{\theta\gamma,\ijbar}=B_{\theta\mu,\jbar}g^{\mu\zetabar}A_{\gamma\zetabar,i}L^{\ijbar}
   										+A_{\theta\mubar, i} g^{\nu\mubar} B_{\nu\gamma, \jbar} L^{\ijbar}
												+H_{\theta\gamma}+V_{\theta\gamma},
\end{align}
where 
\begin{align}\label{expressH}
H_{\theta\gamma}=\frac{\epsilon b}{g}\left(A_{\alpha\betabar,\theta}A_{\mu\zetabar,\gamma}g^{\alpha\betabar}g^{\mu\zetabar}-A_{\gamma\betabar, \alpha} A_{\theta\mubar,\zeta}g^{\zeta\betabar}g^{\alpha\mubar}\right),
\end{align}
and 
\begin{align}\label{expressV}
V_{\theta\gamma}=&-{R_{\gamma\mubar\phi}}^\zeta\Phi_\zeta g^{\phi\betabar} 
		\left(A_{\theta\betabar,\tau}-A_{\theta\betabar,\eta}g^{\eta\xibar}\Phi_{\tau\xibar}\right)
              g^{\alpha\mubar}\Phi_{\alpha\taubar}\\
& -{R_{\theta\mubar\phi}}^\zeta\Phi_\zeta g^{\phi\betabar} 
		\left(A_{\gamma\betabar,\tau}-A_{\gamma\betabar,\eta}g^{\eta\xibar}\Phi_{\tau\xibar}\right)
              g^{\alpha\mubar}\Phi_{\alpha\taubar}\\
&+\frac{\epsilon b}{g}\Phi_\lambda \left({R_{\theta\mubar\zeta}}^\lambda\Phi_{\gamma\betabar\alpha}+{R_{\gamma\mubar\zeta}}^\lambda\Phi_{\theta\betabar\alpha}\right)g^{\alpha\mubar}g^{\zeta\betabar}+{R_{\betabar\gamma\theta}}^\mu\Phi_{\mu\tau}g^{\alpha\betabar}\Phi_{\alpha\taubar}\\
&-g^{\eta\betabar}\Phi_{\tau\betabar}g^{\alpha\mubar}\Phi_{\alpha\taubar}
\left({R_{\theta\mubar\eta}}^\zeta\Phi_{\zeta\gamma}+{R_{\gamma\mubar\eta}}^\zeta\Phi_{\zeta\theta}\right)
		-{R_{\gamma\mubar\theta}}^\zeta\Phi_{\zeta\eta}
		 g^{\eta\betabar}\Phi_{\tau\betabar}g^{\alpha\mubar}\Phi_{\alpha\taubar}\\
&-\left(\nabla_{\gamma}{R_{\theta\mubar\eta}}^\zeta\right)\Phi_{\zeta}
			g^{\eta\betabar}\Phi_{\tau\betabar}g^{\alpha\mubar}\Phi_{\alpha\taubar}
	+\frac{\epsilon b}{g}g^{\alpha\betabar}\left(\nabla_{\gamma}{R_{\betabar\theta\alpha}}^\mu\right)\Phi_\mu
			\label{temp2302022603}\\
&+\frac{\epsilon b}{g}g^{\alpha\betabar}\left({R_{\betabar\gamma\theta}}^\mu\Phi_{\mu\alpha}+
      {R_{\betabar\gamma\alpha}}^{\mu}\Phi_{\mu\theta}
		+{R_{\betabar\theta\alpha}}^{\mu}\Phi_{\mu\gamma}\right)
\end{align}
\begin{remark}
$V$ and $H$ are both symmetric tensors ( (\ref{temp2302022603}) is symmetric in $\theta$ and $\gamma$ because of the second Bianchi identity). $H$ vanishes when $\dim(\MV)=1$, $V$ vanishes when $\omega_0$  is flat, analogous to $U$ and $F$.
\end{remark}

We have the following estimates for $H$ and $V$. There is a constant $\ConstB$, which only depends on the dimension of the K\"ahler manifold $\MV$, so that, when $Q<1$ and (\ref{metricassumption}) is satisfied,
\begin{align}\label{controlH}
|H|\leq \ConstB\frac{\epsilon b}{g}E
\end{align}
\begin{align}\label{controlV}
|V|\leq \ConstB(\max|R|+\max|\nabla R|) \left(T+\sqrt{QTP}+\frac{\epsilon b}{g}(\sqrt{Q}+\sqrt{QE})\right).
\end{align}

We also need to compute $L\overline{B}$. By simply commuting indices, we have
\begin{align}&L^{\ijbar}\overline{B_{\theta\gamma}}_{,\ijbar}
=\overline{L^{j \ibar}\Phi_{\theta\gamma, \ibar j}}=
\overline{
L^{j \ibar}\Phi_{\theta\gamma,  j\ibar}+L^{j \ibar}\left({R_{j\ibar\theta}}^\mu\Phi_{,\mu\gamma}
																				+{R_{j\ibar\gamma}}^\mu\Phi_{,\mu\theta}\right).
}
\end{align}
 $L^{j \ibar}\left({R_{j\ibar\theta}}^\mu\Phi_{,\mu\gamma}
																				+{R_{j\ibar\gamma}}^\mu\Phi_{,\mu\gamma}\right)$ equals
$L^{\beta \alphabar}\left({R_{\beta\alphabar\theta}}^\mu\Phi_{,\mu\gamma}
																				+{R_{\beta\alphabar\gamma}}^\mu\Phi_{,\mu\gamma}\right)$ because of (\ref{tauCurvature0}).
We denote
\begin{align}\label{definitionW}
W_{\thetabar\gammabar}=\overline{V_{\theta\gamma}+L^{\beta \alphabar}\left({R_{\beta\alphabar\theta}}^\mu\Phi_{,\mu\gamma}
																				+{R_{\beta\alphabar\gamma}}^\mu\Phi_{,\mu\theta}\right)}.
\end{align}
Then 
\begin{align}
L^{\ijbar}\overline{B_{\theta\gamma}}_{\ijbar}
=L^{\ijbar}\overline{B_{\gamma\rho}}_{,i}g^{\eta\rhobar}A_{\eta\thetabar, \jbar}
+L^{\ijbar}\overline{B_{\theta\rho}}_{,i}g^{\eta\rhobar}A_{\eta\gammabar, \jbar}
+\overline{H_{\theta\gamma}}
+W_{\thetabar\gammabar}
\end{align}

The estimate of $W$ depends on (\ref{controlV}) and an estimate of 
$L^{\beta \alphabar}\left({R_{\beta\alphabar\theta}}^\mu\Phi_{,\mu\gamma}
	+{R_{\beta\alphabar\gamma}}^\mu\Phi_{,\mu\theta}\right)$
						 in the following.
																				 By (\ref{eq:operatorcoef}), we have
		\begin{align}(L^{\alpha\betabar})=g^{\alpha\overline\theta}\Phi_{\tau\overline{\theta}}g^{\mu\overline{\beta}}\Phi_{\mu\overline\tau}+ \frac{\epsilon b}{g}g^{\abbar}.\end{align}
So for a constant $\ConstBbar$, only depending on the dimension of $\MV$, when  (\ref{metricassumption}) is satisfied, we have
\begin{align}
|L^{\beta \alphabar}\left({R_{\beta\alphabar\theta}}^\mu\Phi_{,\mu\gamma}
													+{R_{\beta\alphabar\gamma}}^\mu\Phi_{,\mu\theta}\right)|
\leq  \ConstBbar \cdot\left(T+\frac{\epsilon b}{g}\right) \cdot\max|R|\cdot \sqrt{Q}
\end{align}

%===========================================================================
\subsection{Equations for $\Phi_{\theta}$}\label{sec:gradient}
By commuting indices in (\ref{equationthetader}), we get
\begin{align}\label{expressS}
L^{\ijbar}\Phi_{\theta, \ijbar}={R_{\zetabar\theta\mu}}^\psi\Phi_{\psi}g^{\mu\betabar}\Phi_{\tau\betabar}g^{\alpha\zetabar}\Phi_{\alpha\taubar}+\frac{\epsilon b}{g}{R_{\betabar\theta\alpha}}^\mu\Phi_{\mu} g^{\alpha\betabar}.
\end{align}
We denote this  by
\[L^{\ijbar}\Phi_{\theta, \ijbar}=S_\theta, \]
and for some constant $\ConstG$, only depending on dimension, 
\begin{align}
|S|\leq \ConstG\cdot (\max|R|)\cdot ( T+ \frac{\epsilon b}{g}) \sqrt{Q},
\end{align}
providing
 (\ref{metricassumption}) is satisfied.
For $\Phi_{\thetabar}$, we have
\begin{align}
L^{\ijbar}\Phi_{\thetabar, \ijbar}&=0
\end{align}
\subsection{Summary}\label{sec:sumcontrol}
To sum up,  we get
\begin{align}
\label{equationA}
L^{\ijbar}A_{\theta\gammabar, \ijbar}=A_{\theta\mubar, i} g^{\zeta\mubar} A_{\zeta\gammabar, \jbar}L^{\ijbar}
+B_{\theta\mu,\jbar}g^{\mu\zetabar}\overline{B_{\zeta\gamma}}_{,i} L^{\ijbar}+F_{\theta\gammabar}+U_{\theta\gammabar},
\end{align}
\begin{align}\label{equationBsum}
L^{\ijbar}B_{\theta\gamma,\ijbar}=B_{\theta\mu,\jbar}g^{\mu\zetabar}A_{\gamma\zetabar,i}L^{\ijbar}
   										+A_{\theta\mubar, i} g^{\nu\mubar} B_{\nu\gamma, \jbar} L^{\ijbar}
												+H_{\theta\gamma}+V_{\theta\gamma},
\end{align}
and 
\begin{align}\label{equationBbarsum}
L^{\ijbar}\overline{B_{\theta\gamma}}_{,\ijbar}=\overline{B_{\theta\mu}}_{,i}g^{\alpha\mubar}A_{\alpha\gammabar,\jbar}L^{\ijbar}+\overline{B_{\gamma\mu}}_{,i}g^{\alpha\mubar}A_{\alpha\thetabar,\jbar}L^{\ijbar}+\overline{H_{\theta\gamma}}+W_{\thetabar\gammabar}.
\end{align}
Providing $Q<1$ and 
\begin{align}
{\frac{1}{2}}(b_{\alpha\betabar})<(g_{\alpha\betabar})<{\frac{3}{2}}(b_{\alpha\betabar} ) 		,		
\end{align}
  there is a constant $C_n$ only depending on $\dim(\MV)$, so that
\begin{align}
|F|+|H|+ |U|+&|V|+|W|+|S|
\leq C_n\cdot (\max|R|+\max|\nabla R|)\frac{\epsilon b}{g} (\sqrt{Q}+\sqrt{QE})\\
&+C_n(\max|R|+\max|R|^2+\max|\nabla R|)(T+\sqrt{QTP})+C_n \frac{\epsilon b}{g} E.
\end{align}
We denote $\CR=C_n\cdot (1+\max|R|+\max|R|^2+\max|\nabla R|)$, then
\begin{align}\label{sumcontrol}
|F|+|H|+ |U|+|V|+|W|+|S|
\leq \CR\left( T+\sqrt{QTP}+\frac{\epsilon b}{g} (\sqrt{Q}+\sqrt{QE}+E)\right).
\end{align}
%%%%%%%%%%%%%%%%%%%%%%%%
%%%%%%%%%%%%%%%%%%%%%%%%%
%%%%%<<<<<<<<section2<<<<<<<section2<<<<section2<<<<section2<<<<section2<<<<section2<
%%%%%<<section2<<<<section2<<<<section2<<<<section2<<<<section2<<<<section2<<<<section2
%%%%+++++++++++++++++++++++++++++++++++++++++++++++++++++++++++++++
%%%%%%>>>>section3>>>section3>>>section3>>>section3>>>section3>>>section3
%%%%%%>>>section3>>>section3>>>section3>>>section3>>>section3
\section{Estimates of Second Order Derivatives of $\Phi$}\label{sec:subharm}
In this section, we estimate the second order derivatives of $\Phi$ by estimating the quantity
\begin{align}\label{eq:Qdef}
Q=\Phi_{\covab}\overline{\Phi_{,\gamma\theta}}g^{\alpha\overline{\theta}}g^{\beta\overline{\gamma}}+
		\Phi_{\abbar}\Phi_{\gamma\overline\theta}g^{\alpha\overline{\theta}}g^{\gamma\overline{\beta}}+
			\Phi_{\alpha}\overline{\Phi_{\beta}}g^{\alpha\overline \beta},
\end{align}
here the covariant derivatives $\Phi_{,\alpha\beta}$ are with respect to the Levi-Civita connection of $\omega_0$. 
First, with computations and the estimate (\ref{sumcontrol}), we get that when some assumptions are satisfied,  
\begin{align}
\label{temp3220220603}
LQ>-\lambda Q,
\end{align}
for a small constant $\lambda$.
This is the main content of section \ref{sec:LQcompute}. Then with (\ref{temp3220220603}), we prove
 Proposition \ref{Proposition1} in section \ref{sec:Qestimate}, which is an apriori estimate showing the value of $Q$ in $\MR\times \MV$ can be controlled by the value of $Q$ on $\partial\MR\times \MV$, which only depends on the boundary value $F$. Then with a continuity argument we prove Theorem \ref{TheoremDisc}.
\subsection{Computation and Estimates for $LQ$}\label{sec:LQcompute}
We denote
\begin{align}
Q_B&=\Phi_{\covab}\overline{\Phi_{,\gamma\theta}}g^{\alpha\overline{\theta}}g^{\beta\overline{\gamma}},
\label{expressQB}\\
Q_A&=\Phi_{\abbar}\Phi_{\gamma\overline\theta}g^{\alpha\overline{\theta}}g^{\gamma\overline{\beta}},
\label{expressQA}\\
		Q_G&=\Phi_{\alpha}\overline{\Phi_{\beta}}g^{\alpha\overline \beta},\label{expressQG}
\end{align}
and do the computation separately.

First, for $Q_B$, we have
\begin{align}
&L Q_B=B_{\alpha\beta, \jbar}\overline{B_{\theta\mu}}_{,i}g^{\alpha\thetabar}g^{\beta\mubar}L^{\ijbar}\\
     +&\left(B_{\alpha\beta, i}-B_{\alpha\eta}g^{\eta\xibar}A_{\beta\xibar,i}-B_{\beta\eta}g^{\eta\xibar}A_{\alpha\xibar,i}\right)\left(\overline{B_{\theta\mu,j}}-A_{\phi\mubar, \jbar}g^{\phi\zetabar}\overline{B_{\theta\zeta}}-A_{\phi\thetabar, \jbar}g^{\phi\zetabar}\overline{B_{\mu\zeta}}\right)g^{\alpha\thetabar}g^{\beta\mubar}L^{\ijbar}\\
&-B_{\alpha\beta}g^{\alpha\zetabar}\overline{B_{\zeta\lambda}}_{,i}g^{\xi\lambdabar}B_{\xi\phi,\jbar}g^{\phi\thetabar}\overline{B_{\theta\mu}}g^{\beta\mubar}L^{\ijbar}-B_{\alpha\beta}g^{\alpha\thetabar}\overline{B_{\theta\mu}}g^{\phi\mubar}B_{\xi\phi,\jbar}g^{\xi\lambdabar}\overline{B_{\lambda\zeta}}_{,i}g^{\beta\zetabar}L^{\ijbar}\\
&+B_{\alpha\beta} g^{\alpha\thetabar }g^{\beta\mubar} \overline{H_{\theta\mu}}
  +B_{\alpha\beta} g^{\alpha\thetabar }g^{\beta\mubar} W_{\thetabar\mubar}
 +H_{\alpha\beta}g^{\alpha\thetabar }g^{\beta\mubar}\   \overline{B_{\theta\mu}}
 +V_{\alpha\beta}g^{\alpha\thetabar }g^{\beta\mubar} \  \overline{B_{\theta\mu}}\\
&-B_{\alpha\beta}g^{\alpha\zetabar}F_{\phi\zetabar}g^{\phi\thetabar}g^{\beta\mubar}\overline{B_{\theta\mu}}
  -B_{\alpha\beta}g^{\alpha\zetabar}U_{\phi\zetabar}g^{\phi\thetabar}g^{\beta\mubar}\overline{B_{\theta\mu}}
\\
  &-B_{\alpha\beta}g^{\beta\zetabar}F_{\phi\zetabar}g^{\phi\mubar}g^{\alpha\thetabar}\overline{B_{\theta\mu}}
  -B_{\alpha\beta}g^{\beta\zetabar}U_{\phi\zetabar}g^{\phi\mubar}g^{\alpha\thetabar}\overline{B_{\theta\mu}}.
\end{align}
The computation is straightforward, we differentiate (\ref{expressQB}) with Leibniz rule and plug in (\ref{equationA}) (\ref{equationBsum}) (\ref{equationBbarsum}), then we note that there are twelve terms, which do not contain $U, V, W, F, H$, cancel with each other. A simplified computation, when $U, V, W, F, H$ all vanish, is in Appendix \ref{sec:Leaf}.
 
 There is a constant $\TCB$, only depending on the dimension, so that
\begin{align}\label{estimateLQB}
LQ_B\geq 
(1-\TCB Q) B_{\alpha\beta, \jbar}\overline{B_{\theta\mu}}_{,i}g^{\alpha\thetabar}g^{\beta\mubar}L^{\ijbar}-\TCB(|W|+|H|+|V|+|F|+|U|)\sqrt{Q},
\end{align}
providing that $Q<1$.

For $Q_A$, by differentiating the expression (\ref{expressQA}) and plugging in (\ref{equationA}), we get
\begin{align}
LQ_A&=2L^{\ijbar}
\left(\Phi_{\alpha\betabar,i}\Phi_{\theta\gammabar}g^{\alpha\gammabar}g^{\theta\betabar}
    		-\Phi_{\alpha\betabar}\Phi_{\theta\gammabar}g^{\alpha\mubar}\Phi_{\nu\mubar,i}
					g^{\nu\gammabar}g^{\theta\betabar}\right)_{\jbar}\\
&=2L^{\ijbar}\left(\Phi_{\alpha\betabar,i}\Phi_{\theta\gammabar,\jbar}g^{\alpha\gammabar}g^{\theta\betabar}
-\Phi_{\alpha\betabar,i}\Phi_{\theta\gammabar}g^{\alpha\etabar}\Phi_{\rho\etabar,\jbar}g^{\rho\gammabar}g^{\theta\betabar}
-\Phi_{\alpha\betabar,i}\Phi_{\theta\gammabar}g^{\alpha\gammabar}g^{\theta\mubar}\Phi_{\rho\mubar, \jbar}g^{\rho\betabar}\right.\\
&\ \ \ \ \ +A_{\alpha\mubar,i}g^{\eta\mubar}A_{\eta\betabar, \jbar}\Phi_{\theta\gammabar}g^{\alpha\gammabar}g^{\theta\betabar}
+B_{\alpha\eta,\jbar}g^{\eta\mubar}\overline{B_{\mu\beta}}_{,i}\Phi_{\theta\gammabar}g^{\alpha\gammabar}g^{\theta\betabar}\\
&\ \ \ \ \ \ -\Phi_{\alpha\betabar,\jbar}\Phi_{\theta\gammabar}g^{\alpha\mubar}\Phi_{\nu\mubar,i}
					g^{\nu\gammabar}g^{\theta\betabar}
-\Phi_{\alpha\betabar}\Phi_{\theta\gammabar,\jbar}g^{\alpha\mubar}\Phi_{\nu\mubar,i}
					g^{\nu\gammabar}g^{\theta\betabar}\\
&\ \ \ \ \ \ +\Phi_{\alpha\betabar}\Phi_{\theta\gammabar}g^{\alpha\rhobar}\Phi_{\eta\rhobar,\jbar}g^{\eta\mubar}\Phi_{\nu\mubar,i}g^{\nu\gammabar}g^{\theta\betabar}
+\Phi_{\alpha\betabar}\Phi_{\theta\gammabar}g^{\alpha\mubar}\Phi_{\nu\mubar,i}
					g^{\nu\etabar}\Phi_{\zeta\etabar,\jbar}g^{\zeta\gammabar}g^{\theta\betabar}\\
&\ \ \ \ \ \  +\Phi_{\alpha\betabar}\Phi_{\theta\gammabar}g^{\alpha\mubar}\Phi_{\nu\mubar,i}
					g^{\nu\gammabar}g^{\theta\etabar}\Phi_{\zeta\etabar,\jbar}g^{\zeta\betabar}			-\Phi_{\alpha\betabar}\Phi_{\theta\gammabar}g^{\alpha\mubar}\Phi_{\nu\etabar,i}
					g^{\zeta\etabar}\Phi_{\zeta\mubar,\jbar}g^{\nu\gammabar}g^{\theta\betabar}\\
&\ \ \ \ \ \ \left. -\Phi_{\alpha\betabar}\Phi_{\theta\gammabar}g^{\alpha\mubar}B_{\nu\zeta,\jbar}
					g^{\zeta\etabar}{\overline{B_{\eta\mu}}}_i g^{\nu\gammabar}g^{\theta\betabar}\right)\\
&\ \ \ \ \ \  +(U_{\alpha\betabar}+F_{\alpha\betabar})\Phi_{\theta\gammabar}g^{\alpha\gammabar}g^{\theta\betabar}-\Phi_{\alpha\betabar}\Phi_{\theta\gammabar}g^{\alpha\mubar}(F_{\nu\mubar}+U_{\nu\mubar})g^{\nu\gammabar}g^{\theta\betabar}.
\end{align}
Then  under the assumption that $Q<1$ and (\ref{metricassumption}), there is a constant $\TCA$, so that
\begin{align}\label{estimateLQA}
LQ_A\geq &A_{\alpha\betabar,i}g^{\alpha\thetabar}A_{\mu\thetabar,\jbar} g^{\mu\betabar}(2-\TCA \sqrt{Q})L^{\ijbar}
  			-\TCA B_{\alpha\beta,\jbar}g^{\alpha\thetabar}\overline{B_{\theta\gamma}}_{,i} g^{\beta\gammabar} \sqrt{Q}L^{\ijbar}-\TCA (|F|+|U|)\sqrt{Q}.
\end{align}

The estimate for $LQ_G$ is more complicated, because we want $T$, defined in (\ref{expressT}),  to be the dominating term. The complete expression for $LQ_G$ is
\begin{align}
LQ_G=&\frac{1}{2}\Phi_{,\alpha i}\Phi_{,\betabar \jbar}g^{\alpha\betabar}L^{\ijbar}
				+\frac{1}{2}\Phi_{\alpha \jbar}\Phi_{\betabar i}g^{\alpha\betabar}L^{\ijbar}
					+S_\alpha \Phi_{\betabar}g^{\alpha\betabar}\\
&+\frac{1}{2}\left(\Phi_{,\alpha i}-2\Phi_{\eta}g^{\eta\gammabar}A_{\alpha\gammabar,i}\right)
      						\left(\Phi_{, \betabar \jbar}-2\Phi_\mubar g^{\theta\mubar}A_{\theta\betabar, \jbar} \right)
							 g^{\alpha\betabar}L^{\ijbar}\\
&+\frac{1}{2}\left(\Phi_{\alpha \jbar}-2\Phi_{\eta}g^{\eta\gammabar}A_{\alpha\gammabar,\jbar}\right)
      						\left(\Phi_{ \betabar i}-2\Phi_\mubar g^{\theta\mubar}A_{\theta\betabar,i} \right)
							 g^{\alpha\betabar}L^{\ijbar}\\
&-\left(\Phi_{\eta}g^{\eta\gammabar}A_{\alpha\gammabar,\jbar}\right)\left(\Phi_\mubar g^{\theta\mubar}A_{\theta\betabar,i} \right)g^{\alpha\betabar}L^{\ijbar}
-\left(\Phi_{\eta}g^{\eta\gammabar}A_{\alpha\gammabar,i}\right)\left(\Phi_\mubar g^{\theta\mubar}A_{\theta\betabar,\jbar} \right)g^{\alpha\betabar}L^{\ijbar}\\
&-\Phi_\alpha\Phi_\betabar g^{\alpha\gammabar}g^{\eta\betabar}\left(A_{\eta\thetabar,i}g^{\mu\thetabar}A_{\mu\gammabar,\jbar}+B_{\eta\mu}g^{\mu\thetabar}\overline{B_{\gamma\theta}}_{, i}+U_{\eta\gammabar}+F_{\eta\gammabar}\right)L^{\ijbar}.
\end{align}
In the following, we will simplify the expression above into an inequality. First, there is a constant $\TCG$, depending only on the dimension of $\MV$, such that
\begin{align}
LQ_G\geq &\frac{1}{2}\Phi_{,\alpha i}\Phi_{,\betabar \jbar}g^{\alpha\betabar}L^{\ijbar}\label{tem318}
				+\frac{1}{2}\Phi_{\alpha \jbar}\Phi_{\betabar i}g^{\alpha\betabar}L^{\ijbar}-\TCG Q(P+\frac{\epsilon b}{g} E)-\TCG  (|S|+|U|+|F|)\sqrt{Q},
\end{align}
providing that the condition (\ref{metricassumption}) is satisfied and $Q<1$.

In the estimate (\ref{tem318}), we want to replace
\[\frac{1}{2}\Phi_{,\alpha i}\Phi_{,\betabar \jbar}g^{\alpha\betabar}L^{\ijbar}
				+\frac{1}{2}\Phi_{\alpha \jbar}\Phi_{\betabar i}g^{\alpha\betabar}L^{\ijbar}\]
by a multiple of $$T=\Phi_{\alpha \tau}\Phi_{\betabar \taubar}g^{\alpha\betabar}
				+\Phi_{\alpha \taubar}\Phi_{\betabar \tau}g^{\alpha\betabar},$$ which is more convenient to use, as we saw in section \ref{sec:sec}.

Because $(L^{\ijbar})\geq (p^{\ijbar})$, we have
\begin{align}
&\frac{1}{2}\Phi_{,\alpha i}\Phi_{,\betabar \jbar}g^{\alpha\betabar}L^{\ijbar}
				+\frac{1}{2}\Phi_{\alpha \jbar}\Phi_{\betabar i}g^{\alpha\betabar}L^{\ijbar}\\
\geq&\frac{1}{2}\Phi_{,\alpha i}\Phi_{,\betabar \jbar}g^{\alpha\betabar}p^{\ijbar}
				+\frac{1}{2}\Phi_{\alpha \jbar}\Phi_{\betabar i}g^{\alpha\betabar}p^{\ijbar}\\
\geq&\frac{1}{2}\left(\Phi_{\alpha \tau}\Phi_{\betabar \taubar}g^{\alpha\betabar}
 												-\Phi_{\alpha \tau}\Phi_{,\betabar \gammabar}g^{\alpha\betabar}g^{\eta\gammabar}\Phi_{\eta\taubar}-\Phi_{,\alpha \rho}\Phi_{\betabar \taubar}g^{\alpha\betabar}g^{\rho\psibar}\Phi_{\tau\psibar}\right)\\
				&+\frac{1}{2}\left(\Phi_{\alpha \taubar}\Phi_{\betabar \tau}g^{\alpha\betabar}-\Phi_{\alpha \taubar}\Phi_{\betabar \rho}g^{\alpha\betabar}g^{\rho\etabar}\Phi_{\tau\etabar}-\Phi_{\alpha \zetabar}\Phi_{\betabar \tau}g^{\alpha\betabar}g^{\theta\zetabar}\Phi_{\theta\taubar}\right)\\
\geq& T\left(\frac{1}{2}-\HCG \sqrt{Q}\right),
\end{align}
for a constant $\HCG$ only depending on the dimension of $\MV$.
 Then plugging the estimate above into (\ref{tem318}), we get that 
\begin{align}\label{estimateLQG}
LQ_G\geq T\left(\frac{1}{2}-\HCG \sqrt{Q}\right)-\TCG Q(P+\frac{\epsilon b}{g} E)-\TCG  (|S|+|U|+|F|)\sqrt{Q},
\end{align}

Put (\ref{estimateLQB}) (\ref{estimateLQA}) and (\ref{estimateLQG}) together, we get that for $\TC=2(\TCB+\TCA+\TCG+\HCG)$
\begin{align}\label{estimateLQ}
LQ\geq \left(\frac{1}{2}-\TC \sqrt{Q}\right)\left(T+P+\frac{\epsilon b}{g}E\right)-\TC(|U|+|V|+|W|+|F|+|H|+|S|)\sqrt{Q},
\end{align}
providing $Q<1$ and the condition (\ref{metricassumption}) is satisfied. Plug in the estimate (\ref{sumcontrol}), we get, for $\const=(\TC+1)(\CR+1)2^{n+2}$
\begin{align}\label{estimateresult}
LQ\geq (\frac{1}{2}-\const \sqrt{Q})(T+P+\frac{\epsilon b}{g})-\const \epsilon Q.
\end{align}
So, if $Q\leq \frac{1}{4\const^2}$,
\begin{align} \label{LQinequal}
LQ\geq  -\const \epsilon  Q.
\end{align}
%%%%%%%%%%%%%%%%%%%%%%%%%%%%%%%%%%%%%%%%%%%%
%%%%%%%===================
\subsection{Estimating $Q$}\label{sec:Qestimate}
With the inequality (\ref{LQinequal}), we can derive the following apriori estimate for $Q$. Here we denote the diameter of  $\MR$ by $\diameterR$, and assume that $\MR$ is contained in the ball $\{|\tau|\leq \diameterR\}\subset \EC$.
\begin{Proposition}\label{Proposition1}
	There is a constant $\delta$ depending on the dimension, and the curvatures and the covariant derivatives of the curvatures of the manifold $\MV$, so that for a solution $\Phi$ to Problem \ref{ProblemDisc}, if  the function $Q$, defined by (\ref{eq:Qdef}), satisfies
$Q\leq \delta^2,$ and $\epsilon<\frac{\delta}{\diameterR^2}$,  then 
\begin{align}\label{boundarycontrol}
Q\leq 2\max _{\partial\MR\times \MV}Q,\text{\ \ \ \ \ in\  } \MR\times \MV.
\end{align}
\end{Proposition}
\begin{proof}
	First, we require that $\delta<\frac{1}{4}$, then $Q<\delta^2$ implies $ Q<\frac{1}{10}$. So by (\ref{eq:Qdef}),  
	\begin{align}
	\Phi_{\abbar}\Phi_{\gamma\overline\theta}g^{\alpha\overline{\theta}}g^{\gamma\overline{\beta}}<\frac{1}{10}.\label{condition0627}
	\end{align}
At any point, we can choose local coordinates, so that 
\begin{align}
	\Phi_{\alpha\betabar}=\lambda_{\alpha}\delta_{\alpha\beta}, \ \ b_{\alpha\betabar}=\delta_{\alpha\beta}.
\end{align}
Then condition (\ref{condition0627}) implies that 
\begin{align}
	|\lambda_\alpha|<\frac{1}{3}, \ \ \text{ for all }\alpha=1, ..., n.
\end{align}This implies (\ref{metricassumption}).

Let $$u(\tau)=\cos\left(\frac{\pi \text{Re}(\tau)}{4\diameterR}\right)\cos\left(\frac{\pi 
	\text{Im}(\tau)}{4\diameterR}\right),$$
then 
\begin{align}
	\label{equationu0627}u_{\tau\taubar}=-\frac{\pi^2}{32\diameterR^2}u.
	\end{align}
We consider $u$ as a function defined on $\MR\times \MV$, by letting $u(\tau,z)=u(\tau)$. Then (\ref{equationu0627}) becomes
\begin{align}\label{temp0627}
	u_{\ijbar}L^{\ijbar}=-\frac{\pi^2}{32\diameterR^2}u.
\end{align} At an interior maximum point of $\frac{Q}{u}$, we have
\[Q_i=Q\frac{u_i}{u},  \text{ for any }i\in\{0,1,... n\},\]and by (\ref{LQinequal}) (\ref{temp0627})
\[0\geq \left(\frac{Q}{u}\right)_{\ijbar}L^{\ijbar}=\left(\frac{Q_{\ijbar}}{u}-\frac{Q u_{\ijbar}}{u^2}\right)L^{\ijbar}\geq \left(\frac{\pi^2}{32\diameterR}-\const\epsilon\right)\frac{Q}{u}.\]
If $\epsilon\leq \frac{\pi^2}{32\diameterR\const}$, then $\frac{Q}{u}$ cannot achieve interior maximum, except that $Q\equiv 0$.
  Thus, 
\[\frac{Q}{u}\leq \max_{\partial\MR\times \MV} \frac{Q}{u}, \text{\ \ \ \ \ in \ \ \ }\MR\times \MV.\] Because
 $1\geq u\geq\frac{1}{2}$, we  have 
\[ Q\leq 2 \max_{\partial\MR\times \MV} Q, \text{\ \ \ \ \ in \ \ \ }\MR\times \MV.\] 
So, we can let $\delta=\min\left\{\frac{\pi^2}{32 \const}, \frac{1}{4}\right\}$.
\end{proof}
Then with the previous proposition we can get an estimate for $Q$. 
%First, we would need the following lemma \begin{lemma}\label{LemmaMonotone}\end{lemma}
For $\lambda\in[0,1]$, let $\Phi^\lambda$ be the solution to Problem \ref{ProblemDisc}, with the boundary value condition being replaced by
\[\Phi^\lambda=\lambda \cdot F, \text{\ \ \ \ \ on \ }\partial \MR\times V.\]
Define 
\[g_{\alpha\betabar}^\lambda=b_{\alpha\betabar}+\Phi_{\alpha\betabar}^\lambda, \] and denote its inverse by 
$g_{\lambda}^{\alpha\betabar}$. Then we define 
\begin{align}\label{eq:Qlambdadef}
Q^\lambda
=\Phi^\lambda_{\covab}\overline{\Phi^\lambda_{,\gamma\theta}}g_\lambda^{\alpha\overline{\theta}}g^{\beta\overline{\gamma}}_\lambda+
		\Phi^\lambda_{\abbar}\Phi^\lambda_{\gamma\overline\theta}g^{\alpha\overline{\theta}}_\lambda g^{\gamma\overline{\beta}}_\lambda+
			\Phi^\lambda_{\alpha}\overline{\Phi^\lambda_{\beta}}g^{\alpha\overline \beta}_\lambda.
\end{align}
For each $(\tau, z)\in \MR\times\MV$, $Q^{\lambda}(\tau, z)$ is a non-decreasing function of $\lambda$, this can be proved by taking $\partial_\lambda$ derivatives of $Q^\lambda$. We assume that 
\begin{align}
\max_{\partial \MR\times\MV} Q\leq \frac{\delta^2}{2},
\end{align}
where the $\delta$ is from Proposition \ref{Proposition1}.
Then this implies, by the non-decreasing property, for each $\lambda\in [0,1]$, 
\begin{align}
\max_{\partial \MR\times\MV} Q^\lambda\leq  \frac{\delta^2}{2},
\end{align}
 We assume that $F\in C^{\infty}(\partial \MR\times \MV)$ and $\epsilon>0$, so all second order derivatives of $\Phi^{\lambda}$ are continuous with respect to $\lambda$. This can be proved with implicit function theorem. In particular, $\max_{\MR\times \MV}Q^\lambda$ is a continuous function of $\lambda$.  So, if $$\max_{ \MR\times \MV}Q^1=\max_{ \MR\times \MV}Q>2 \max_{\partial \MR\times \MV}Q,$$ there is a $\lambda\in(0,1)$, so that
\begin{align}
\delta^2\geq \max_{ \MR\times \MV}Q^\lambda>2 \max_{\partial \MR\times \MV}Q,
\end{align}because $Q^0=0.$
But, by Proposition \ref{Proposition1}, $\delta^2\geq \max_{ \MR\times \MV}Q^\lambda$ implies
\[\max_{ \MR\times \MV}Q^{\lambda}\leq 2\max_{\partial \MR\times \MV}Q^{\lambda}\leq 2\max_{\partial \MR\times \MV}Q,\]
which is a contradiction.
\begin{figure}[h]
	\centering
	\includegraphics[height=4.4cm]{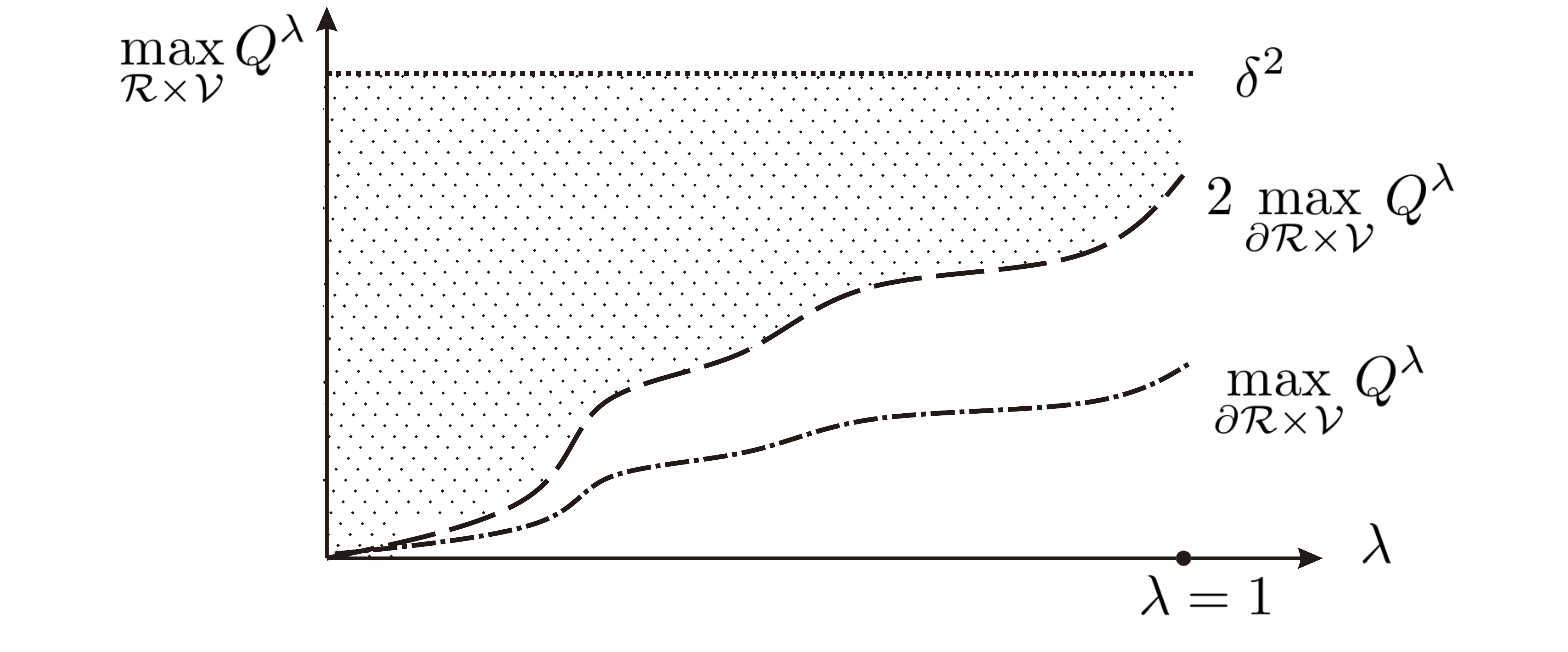}
	\caption{ Continuity Argument}
	\label{fig:Method_of_Continuity}
\end{figure}

The argument above can be illustrated by Figure \ref{fig:Method_of_Continuity}. By Proposition \ref{Proposition1}, for no $\lambda\in (0,1)$, $(\lambda, \max_{\MR\times \MV}Q^\lambda)$ can stay in the shadowed area. This is because Proposition \ref{Proposition1}
 says if $Q^\lambda<\delta^2$ in $\MR\times \MV$, then $\max_{\MR\times \MV}Q^\lambda\leq 2 \max_{\partial\MR\times \MV}Q^\lambda$.  
Therefor, since $\max_{\MR\times\MV}Q^\lambda$ is a continuous function of $\lambda$, with $\max_{\MR\times\MV}Q^0=0$, the curve
\begin{align}
	\{(\lambda, \max_{\MR\times\MV}Q^\lambda)|\lambda\in[0,1]\}
\end{align} has to stay below the shadowed area. So, for all $\lambda\in[0,1]$, 
 \[\max_{\MR\times\MV}Q^\lambda\leq 2 \max_{\partial\MR\times \MV}Q^\lambda\leq 2 \max_{\partial\MR\times \MV}Q^1\leq \delta^2. \]
 
As a conclusion, if for any $\tau\in\partial\MR$ the second order derivatives of $F(\tau, \ast)$ are small enough, so that \begin{align}
\max_{\partial \MR\times\MV} Q\leq \frac{\delta^2}{2},
\end{align} and $\epsilon<\frac{\delta}{10 \diameterR^2}$, we have $Q<{\delta}^2\leq \frac{1}{10}$ and so
\begin{align}
{\frac{1}{2}}(b_{\alpha\betabar})<(g_{\alpha\betabar})<{\frac{3}{2}}(b_{\alpha\betabar} ) 	.
\end{align}
Theorem \ref{TheoremDisc} is proved.
\begin{remark}
The $\theoremdelta$ of Theorem \ref{TheoremDisc} can be chosen as
\begin{align}
\theoremdelta=\min\left\{\frac{1}{C\cdot (1+\max|R|+\max|R|^2+\max|\nabla R|)}, 1/16\right\},
\end{align}
for some constant $C$ only depending on the dimension of the manifold $\MV$. 
\end{remark}

Now given two function $\varphi_1, \varphi_0\in \MH$, with small $C^2$ norm, and $\epsilon\in(0,\frac{\hat\delta}{10 \diameterR^2})$, we can use the theory of non-degenerate Monge-Amp\`ere equation, and get a solution $\Phi^\epsilon$ to Problem \ref{ProblemDisc}, with $\MR=\{\tau|1<|\tau|<e\}$, and 
\begin{align}
\Phi^\epsilon=\varphi_0, \text{\ on\ \ }\{|\tau|=1\}\times \MV;\\
\Phi^\epsilon=\varphi_1, \text{\ on\ \ }\{|\tau|=e\}\times \MV.
\end{align}
With Theorem \ref{TheoremDisc}, we know 
\begin{align}
\omega_0+\sqrt{-1}\ddbar\Phi^\epsilon(\tau, \ast)>\frac{1}{2}\omega_0, \ \ \ \ \text{ for any $\tau\in \MR$}.
\end{align}
When $\epsilon\rightarrow 0$, the sequence of $\Omega_0-$plurisubharmonic functions $\Phi^\epsilon$ monotone increase, and converge to a $\Omega_0-$plurisubharmonic function $\Phi$, in $C^1$ norm. Then $\Phi(e^\zeta, \ast)$ is a solution to Problem \ref{ProblemGeodesic} in the weak sense and it satisfies
\begin{align}
\omega_0+\sqrt{-1}\ddbar\Phi^\epsilon(e^\zeta, \ast)>\frac{1}{2}\omega_0, \ \ \ \ \text{ for any $\zeta\in \MS$},
\end{align}
weakly.  Theorem \ref{LowerBdGeodesic} is proved.
%%%%%%%
\appendix
%%%%%%%%%%%%%%%%%%%%%%%%%%%%%%%%x
\section{Estimates with 1-Dimensional Flat Torus}\label{sec:app}
In this appendix, we prove some estimates for solutions to Problem \ref{ProblemDisc},  in the situation that the K\"ahler manifold $(V, \omega_0)$ is a flat torus. The torus under consideration is $T=\EC/\sim$, where 
\[z\sim z+1, \ z\sim z+\lambda,\]
for $\lambda\in \EC$ and $\lambda\neq 1$,
 the metric on the torus is $\omega_0=\sqrt{-1}dz\wedge\dzbar$. On such a flat torus, we define
\begin{definition}
For a function $\varphi\in C^2(T)$, we say $\varphi$ is $\omega_0-$convex if 
\begin{align}\label{omegaconvexitydef}
\frac{|\varphi_{zz}|}{1+\varphi_{\zzbar}}<1,  \text{\ \ \ and   \ \ \ }1+\varphi_{\zzbar}>0.
\end{align}
\end{definition}
\begin{remark}Condition (\ref{omegaconvexitydef}) is equivalent to 
 \begin{align}
D^2\varphi+dz\otimes \dzbar+\dzbar\otimes dz>0,
\end{align}and it coincides with the concept of $g-$convexity defined in \cite{GTo}, if we let the $(V, \nabla, g)$ in \cite{GTo} be $(T, \partial, dz\otimes \dzbar+\dzbar\otimes dz)$.
\end{remark}

The results of this appendix are: Theorem \ref{preserveconvexity}, which says that if, for any $\tau\in\partial \MR$, $F(\tau, \ast)$ is $\omega_0-$convex, then for the solution $\Phi$ to Problem \ref{ProblemDisc}, $\Phi(\tau, \ast)$ is $\omega_0-$convex, for any $\tau\in \MR$; Theorem \ref{metriclowerbd},  if for all $\tau\in\partial \MR$, $F(\tau, \ast)$ is $\omega_0-$convex, then for all $\tau\in \MR$, the metric
\[\omega_0+\sqrt{-1}\ddbar\Phi(\tau, \ast)\] has a lower bound, independent of $\epsilon$; and Theorem \ref{upperbound},  for the solution $\Phi$ to Problem \ref{ProblemDisc}, $|\Phi_{zz}|$ and $\Phi_{z\overline z}$ has a very precise upper bound, independent of $\epsilon$. These results will be generalized to higher dimensional case in an upcoming paper. 

We start with
\begin{theorem}\label{preserveconvexity}Given $F\in C^{\infty}(\partial\MR\times T)$, if for all $\tau\in \partial \MR$, $F(\tau, \ast)$ is $\omega_0-$convex, then for $\Phi$, the solution to Problem \ref{ProblemDisc}, with boundary value $F$, $\Phi(\tau, \ast)$ is $\omega_0$-convex, for all $\tau\in\MR$.
\end{theorem}
\begin{proof}In the following, we denote
\begin{align}
b=\Phi_{zz},\ \ \  \ \ \ a=\Phi_{\zzbar}.
\end{align}
With $(V, \omega_0)$ being a flat torus, the Monge-Amp\`ere equation becomes 
\begin{align}
(1+\Phi_{\zzbar})\Phi_{\tautaubar}-\Phi_{\tau\zbar}\Phi_{z\taubar}=\epsilon.
\end{align} Computations in section \ref{sec:sec} give equations for $a, b$:
\begin{align}\label{equationa}
h^{\ijbar}a_{\ijbar}=\left(\frac{a_ia_{\jbar}}{1+a}+\frac{b_\jbar\bbar_i}{1+a}\right)h^{\ijbar},
\end{align}
\begin{align}\label{equationb}
h^{\ijbar} b_{\ijbar}=h^{\ijbar}\frac{2b_\jbar \bbar_i}{1+a}.
\end{align}
\begin{remark}
Comparing with section \ref{sec:sec}, the computation here is much simpler.  $H, F, U, V, W$ all vanish, because the dimension of $\MV$ is 1, and $\omega_0$ is flat.
\end{remark}
In this appendix, we define 
\begin{align}\label{definitionQappendix}
Q=\frac{b\bbar}{(1+a)^2},
\end{align}
and show that for a very large constant $K$, 
\begin{align}h^{\ijbar}\left(e^{KQ}\right)_{\ijbar}\geq 0.\end{align} 
Here,
\begin{align}
(h_{i\jbar})=\left(\begin{array}{cc}\Phi_{\tau\taubar}& \Phi_{\tau\zbar}\\
																								\Phi_{	z\taubar}&1+\Phi_{z\zbar}\end{array}\right),
\end{align}
and
\begin{align}
(h^{ j\ibar})=\frac{1}{\epsilon}\left(\begin{array}{cc}1+\Phi_{z\zbar}& -\Phi_{\tau\zbar}\\
																									-\Phi_{z\taubar}&\Phi_{\tau\taubar}\end{array}\right),
\end{align}
where $i$ is the row index and $j$ is the column index. 
Direct computation gives,
\begin{align}
&h^{\ijbar}\left(e^{KQ}\right)_{\ijbar}\\
=&Ke^{KQ}h^{\ijbar}\left[\left(\frac{b_i}{a+1}-\frac{2a_i b}{(1+a)^2}\right)
																\left(\frac{\bbar_\jbar}{a+1}-\frac{2a_\jbar \bbar}{(1+a)^2}\right)
			+\frac{b_{\jbar}\bbar_i}{(1+a)^2}(1-2Q)\right.\\
    &\ \ \ \ \ \ \ \ \ \ \ \left. +K\left(\frac{b_i\bbar}{(1+a)^2}+\frac{b\bbar_i}{(1+a)^2}-\frac{2Qa_i}{1+a}\right)
    \left(\frac{\bbar_\jbar b}{(1+a)^2}+\frac{\bbar b_\jbar}{(1+a)^2}-\frac{2Qa_\jbar}{1+a}\right)\right].
\end{align}
We denote 
\begin{align}
b_i-\frac{2a_ib}{1+a}=\beta_i,
\end{align}
and
\begin{align}
M=\left(\begin{array}{cc}1+KQ& K\frac{\bbar^2}{(1+a)^2}\\
											K\frac{b^2}{(1+a)^2}		& 1-2Q+KQ\end{array}\right),
\end{align}
then
\begin{align}\label{temp0604a14}
h^{\ijbar}\left(e^{KQ}\right)_{\ijbar}=Ke^{KQ}h^{\ijbar}(\beta_i, \bbar_i) M (\overline{\beta_j}, b_\jbar)^T.
\end{align}
We want to choose a large enough $K$ so that $M\geq 0$, which implies $\text{(\ref{temp0604a14})}\geq 0$. First, we need $K>2$, which implies that the diagonal elements of $M$ are positive. Then we make $K$ larger, so that $$\det(M)=2K\left[-Q^2+(1-\frac{1}{K})Q+\frac{1}{2K}\right]$$ is also positive. As $K\rightarrow \infty$, two roots of 
\begin{align}\label{Kpolynomial}-Q^2+(1-\frac{1}{K})Q+\frac{1}{2K}
\end{align} converge to $0$ and $1$. We need to choose $K$ large enough, so that the larger one of two roots
\begin{align}\label{Kcondition}\sigma_2=\frac{1-\frac{1}{K}+\sqrt{1+\frac{1}{K^2}}}{2}>\max_{\partial \MR\times T}Q.
\end{align} Because  $\sigma_2<1$, for any $K>0$, so (\ref{Kcondition}) can only be achieved under the condition that
\begin{align}
\max_{\partial \MR\times T}Q<1. \end{align}
Since the smaller one of two roots of (\ref{Kcondition})
\begin{align}\sigma_1=\frac{1-\frac{1}{K}-\sqrt{1+\frac{1}{K^2}}}{2}
\end{align}  is always negative,  we have showed that if $0\leq Q<\sigma_2$, in $\MR\times \MV$, then $M\geq 0$ and
\begin{align}
h^{\ijbar}\left(e^{KQ}\right)_{\ijbar}\geq 0.
\end{align}Because $e^{KQ}$ is a strictly monotone increasing function of $Q$, for positive $K$, the maximum principle for Laplace equation tells 
\begin{align}
Q\leq \max_{\partial\MR\times T}Q, \text{\ \ \ \ \ in \ }\MR\times T.\label{Prop1AppendixA}
\end{align}
	The estimate above is an analog of Proposition \ref{Proposition1}. Form here, we can use a continuity argument to show that if 
\begin{align}Q<1, \text{\ \ \ \ on \ }\partial\MR\times T, \label{Qsmaller1appendixA}
	\end{align}
then
\[Q<1, \text{\ \ \ \ in \ } \MR\times T. \]The argument is in the following. 
Similar to the proof in section \ref{sec:Qestimate}, for $\lambda\in [0,1]$, let $\Phi^\lambda$ be the solution to Problem \ref{ProblemDisc}, with the boundary condition (\ref{ProblemDiscBdCondidtion}) being replaced by 
\begin{align}
	\Phi^\lambda=\lambda  F, \ \ \ \text{ on }\partial \MR\times T.
	\end{align}
	Define 
	\[Q^\lambda=\frac{|\Phi^\lambda_{zz}|^2}{(1+\Phi^\lambda_{\zzbar})^2}.\]
	By taking $\partial_\lambda$ derivatives, we can show that for any fixed $(\tau, z)\in \partial\MR\times T$, $Q^\lambda$ is a non-decreasing function of $\lambda$. This implies, $\max_{\partial\MR\times T} Q^\lambda$ is a non-decreasing function of $\lambda$, so, if (\ref{Qsmaller1appendixA}) is satisfied, then  for any $\lambda\in [0,1]$, 
	\begin{align}
	\max_{\partial\MR\times T} Q^\lambda<1.
	\end{align}
We choose $K$ large enough so that 
\begin{align}\max_{\partial\MR\times T}Q<\sigma_2<1.
	\end{align}
 \begin{figure}[h]
 	\centering
 	\includegraphics[height=4.4cm]{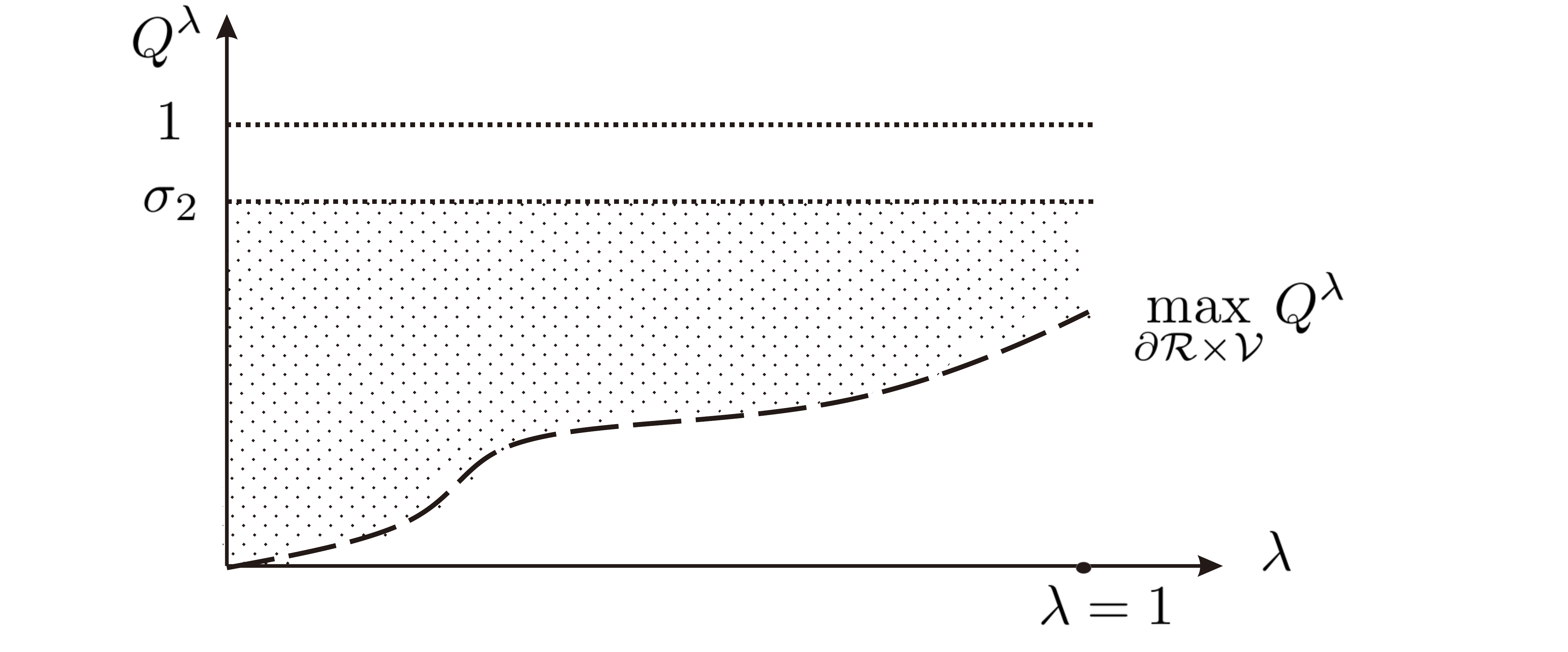}
 	\caption{Continuity Method with $\MV$ being a Torus}
 	\label{fig:Method_of_ContinuityAppendix}
 \end{figure}
Then, by estimate (\ref{Prop1AppendixA}), for any $\lambda\in (0,1)$, 
\begin{align}
	\max_{\MR\times T} Q^\lambda\notin(\max_{\partial \MR\times T}Q^\lambda, \sigma_2), \label{Rangeof }
\end{align}
i.e. as illustrated by Figure \ref{fig:Method_of_ContinuityAppendix}, $(\lambda, \max_{\MR\times T}Q^\lambda)$ cannot stay in the dotted area. Since  $\max_{\MR\times T}Q^\lambda$ is a continuous function of $\lambda$, and $\max_{\MR\times T}Q^0=0$, we know 
\begin{align}
	\max_{\MR\times T}Q=\max_{\MR\times T}Q^1\leq \sigma_2<1.
\end{align}
\end{proof} 

In the same way, we have estimates
\begin{align}
\frac{|\Phi_{zz}-\gamma|^2}{(1+\Phi_{z\zbar})^2}<1, \label{tempA21June42022}
\end{align}with $\gamma\in \EC$ being small enough. By using these estimates,  we can get the following metric lower bound estimate independent of $\epsilon$.
\begin{theorem}\label{metriclowerbd}Given $F\in C^{\infty}(\partial\MR\times T)$, if for all $\tau\in \partial \MR$, $F(\tau, \ast) $ is $\omega_0$-convex, then there is a constant $\delta>0$, so that the solution $\Phi$ to Problem \ref{ProblemDisc} with boundary value $F$ satisfies
\begin{align}
\partial_{\zzbar}\Phi(\tau, \ast)+1>\delta, \text{\ \ \ \ \ for all }\tau\in\MR.
\end{align}
$\delta$ only depends on the second order derivatives of $F(\tau, \ast)$, for all $\tau $ in $  \MR$.
\end{theorem}
\begin{proof}
For all $\gamma\in\EC$, define 
\[Q_{\gamma}=\frac{|\Phi_{zz}-\gamma|^2}{(1+\Phi_{\zzbar})^2}.\]
Then computations show that
\[h^{\ijbar}\left(e^{KQ_\gamma}\right)_{\ijbar}\geq 0,\]
providing that $K$ is big enough and $Q_\gamma<1$. The computations are the same as that of Theorem \ref{preserveconvexity}, we only need to replace $b$ by $\Phi_{zz}-\gamma$. This is because equations (\ref{equationa}) and (\ref{equationb}) only contain derivatives of $b$, so adding a constant to $b$ does not affect the equations.

Because for each $\tau\in \partial \MR$, $F(\tau, \ast)$ is $\omega_0-$convex, we have
\[\max_{\partial \MR\times T}\frac{|F_{zz}|^2}{(1+F_{\zzbar})^2}<1.\]
Since $\partial\MR\times T$ is a compact set, we can find  $\delta>0$, only depending on 
\[\max_{\partial \MR\times T}\frac{|F_{zz}|^2}{(1+F_{\zzbar})^2}\ \ \ \ \ \text{and\ }\ \ \ 
\min_{\partial \MR\times R}|1+F_{\zzbar}|,
\]
so that for all $\gamma\in\EC$, with $|\gamma|\leq\delta$, 
\[\max_{\partial \MR\times T}\frac{|F_{zz}-\gamma|^2}{(1+F_{\zzbar})^2}<1.\]
Then, using a continuity method argument, similar to that of section \ref{sec:Qestimate} or the proof of Theorem \ref{preserveconvexity}, we know
\[\max_{ \MR\times T}\frac{|\Phi_{zz}-\gamma|^2}{(1+\Phi_{\zzbar})^2}<1, \ \ \ \ \text{for all }\gamma, \text{ with }|\gamma|\leq\delta.\]
In particular, we have
\begin{align}|\Phi_{zz}-\delta|\leq 1+\Phi_{\zzbar},\ \ \ \ \text{ in } \MR\times T, \label{tempjune4bx1}\end{align}
\begin{align}|\Phi_{zz}+\delta|\leq 1+\Phi_{\zzbar},\ \ \ \ \text{ in } \MR\times T.\label{tempjune4bx2}\end{align}
Adding (\ref{tempjune4bx1}) and (\ref{tempjune4bx2}) together, we get
\[1+\Phi_{\zzbar}>\delta,\ \ \ \ \text{ in } \MR\times T.\]
\end{proof}
The results and ideas in Theorem \ref{preserveconvexity} and Theorem \ref{metriclowerbd} can be illustrated by Figure \ref{fig:PositionofBoundaryCondition}. First, we define 
\begin{align}
	\ML=\{(b,a)|b\in \EC, a\in(-1, \infty)\}.
\end{align}
 For $\tau\in\MR$, we define a jet map,  $J_\tau:T\rightarrow \ML,$
\[J_\tau(z)=(\Phi_{zz}(\tau, z), \Phi_{\zzbar}(\tau, z)),\]and $J_\tau$ can also be considered as a map from $ \MR\times T$ to $\ML,$
\[J(\tau, z)=J_{\tau}(z).\]
For each $\gamma\in \EC$, define a cone
\[C_{\gamma}=\{(b, a)| b\in \EC, a>-1, |b-\gamma|< a+1\}.\]
In particular, $C_0=\{(b,a)| b\in \EC, a>-1, |b|<a+1\}$.  Theorem \ref{preserveconvexity} says that if 
\[\text{Image($J_\tau$)}\subset C_0,\ \ \  \ \ \ \text{for any $\tau\in\partial \MR$},\] then 
\[\text{Image($J_\tau$)}\subset C_0,\ \ \ \ \ \ \ \text{for any $\tau\in \MR$}.\]
 Theorem \ref{metriclowerbd} says that since \[J(\partial\MR\times V)=\bigcup_{\tau\in\partial\MR}\text{Image($J_\tau$)}\subset C_0,\] there will be a positive distance in between $J(\partial\MR\times V)$ and $\partial C_0$ then,  for $\gamma$ with $|\gamma|\leq \delta$,
\[\text{Image($J_\tau$)}\subset C_\gamma,\ \ \ \ \ \ \ \text{for any $\tau\in \partial\MR$}\]
and so
\[\text{Image($J_\tau$)}\subset C_\gamma,\ \ \ \ \ \ \ \text{for any $\tau\in \MR$}.\]
This is equivalent to
\[J(\MR\times V)=\bigcup_{\tau\in \MR}\text{Image($J_\tau$)}\subset\bigcap_{\gamma, |\gamma|\leq \delta}C_\gamma,\]
which gives the metric lower bound, because 
\[\bigcap_{\gamma, |\gamma|\leq \delta}C_\gamma\subset \{a+1>\delta\}.\]
\begin{figure}[h]
\centering
\includegraphics[height=5.4cm]{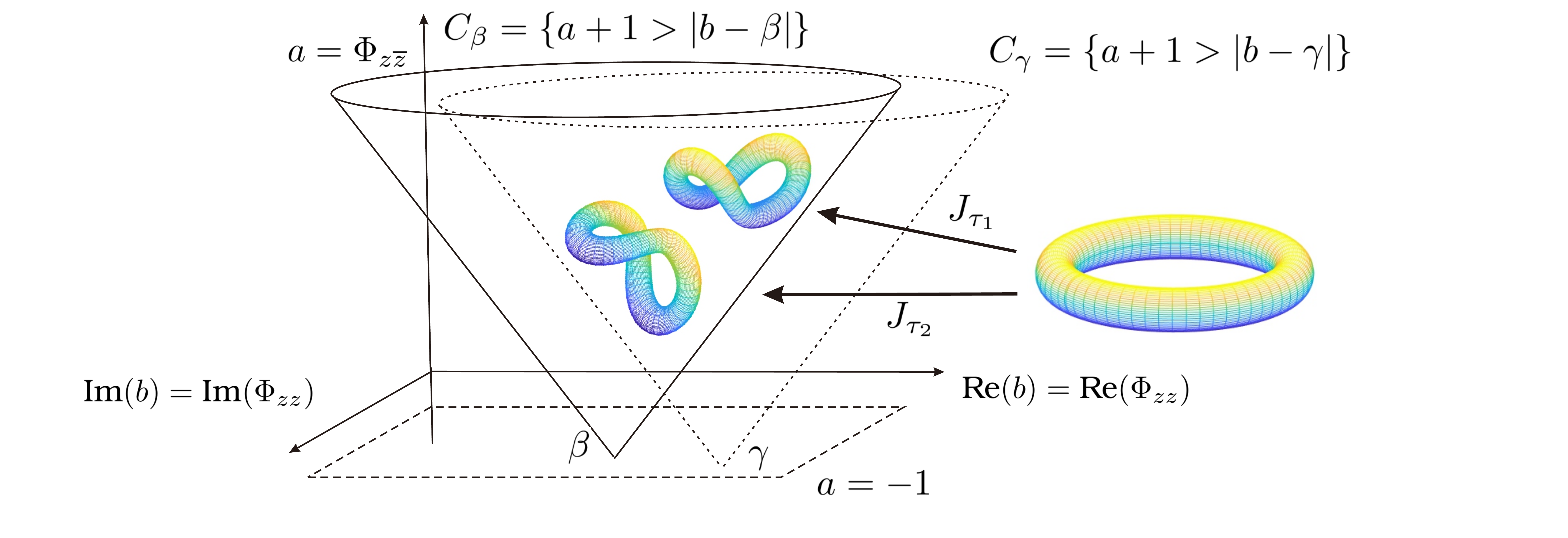}
\caption{Jet Maps}
\label{fig:PositionofBoundaryCondition}
\end{figure}

\begin{remark}
The construction of the function $Q$ is inspired by two facts. First, a quadratic polynomial 
\[P(z)=(1+a)z\zbar+\frac{1}{2}\left(bz^2+\overline{bz^2}\right)\]
is convex if and only if $|b|<1+a$. Second,  the space $\ML$ can be given a Lorentzian metric 
\begin{align}\label{LorentzPoincareMetric}
ds^2=\frac{-da^2+db\overline{db}}{(1+a)^2}. \end{align}With this metric, the cone $C_0$ is a convex set
which means that every geodesic  with ends in $C_0$ stays in $C_0$, this is explained in \cite{Nomizu} and the appendix of \cite{HuIMRN}. It's also an interesting fact that  with the metric (\ref{LorentzPoincareMetric}) the complement of $C_0$, which we denote by $C_0^c,$ is also convex, in the sense that if two ends of a geodesic segment stay in $C_0^c$, then this geodesic  stays in $(C_0)^c$. This fact leads to the following theorem.
\end{remark}
%It's also an interesting fact, that if $Q>1$, we can choose $P\in \ER^+$, large enough, so that
%\[h^{\ijbar}\left(e^{-PQ}\right)_{\ijbar}\geq 0.\]And this would lead to the following upper bound estimate for $|\Phi_{zz}|$ and $\Phi_{\zzbar}$.
\begin{theorem}\label{upperbound}Given $F\in C^{\infty}(\partial\MR\times T)$, let 
\begin{align}S=\max_{\partial\MR\times T}\left(F_{z\zbar}+1+|F_{zz}|\right)
\label{DefS}
\end{align}then for the solution $\Phi$ to Problem \ref{ProblemDisc}, we have 
\[|\Phi_{zz}|\leq S-1-\Phi_{\zzbar},\ \ \ \ \ \ \text{ in\ \  }\MR\times T.\]
\end{theorem}
\begin{remark}
This implies $|\Phi_{zz}|\leq S$ and $\Phi_{\zzbar}\leq S-1$.
\end{remark}
\begin{proof} By the definition of $S$ (\ref{DefS}), for any $\eta\in \EC$, with $|\eta|>S$, 
\[\frac{|F_{zz}-\eta|}{1+F_{\zzbar}}>1.\]Computations analogous to those for Theorem \ref{preserveconvexity}  give us:
\begin{align}
&h^{\ijbar}\left(e^{-PQ_{\eta}}\right)_{\ijbar}=Pe^{-PQ_\eta} h^{\ijbar}(\beta_i, \bbar_i)N(\overline{\beta_j}, b_{\jbar})'\frac{1}{(a+1)^2},
\end{align}
where we denote 
\[\beta_i=b_i-\frac{2a_i(b-\eta)}{1+a},\]
and 
\begin{align}
N=\left(\begin{array}{cc}P Q_{\eta}-1&P\frac{{\overline{(b-\eta)}}^2}{(1+a)^2}\\
											P\frac{{{(b-\eta)}}^2}{(1+a)^2}	& (P+2)Q_\eta-1\end{array}\right).
\end{align}
If we choose $P>1$, then since $Q_\eta>1$, we get the diagonal elements of $N$ are positive. We also need to make $\det N>0$, which  then implies $N>0$. 
\[\det N=2P\left(Q_\eta^2-(1+\frac{1}{P})Q_\eta+\frac{1}{2P}\right).\]
We choose $P$ large, so that the bigger one of two roots of 
$Q_\eta^2-(1+\frac{1}{P})Q_\eta+\frac{1}{2P}$
\[\sigma_2=\frac{1}{2}(1+\frac{1}{P}+\sqrt{1+\frac{1}{P^2}})< \min_{\partial \MR\times T}\frac{|F_{zz}-\eta|^2}{(1+F_{\zzbar})^2}.\]
Then, for $Q_\eta>\sigma_2$, we have
\[h^{\ijbar}\left(e^{-PQ_{\eta}}\right)_{\ijbar}\geq 0.\]Using an analogous continuity method argument in section \ref{sec:Qestimate},  we get 
\[\frac{|\Phi_{zz}-\eta|^2}{(1+\Phi_{\zzbar})^2}=Q_\eta\geq \min_{\partial \MR\times T} Q_\eta=\min_{\partial \MR\times T}\frac{|F_{zz}-\eta|^2}{(1+F_{\zzbar})^2}>1.\]
This says that for any $\eta$, with $|\eta|>S$, 
\[J(\MR\times T)\subset C_\eta^c.\] So
\[J(\MR\times T)\subset \bigcap_{\eta, |\eta|>S}C_\eta^c=\{(b, a)| a+1>0, |b|<S-a-1\}.\]
Now we proved 
\[|\Phi_{zz}|\leq S-\Phi_{\zzbar}-1.\]
\end{proof}

\section{Computation Along the Leaf}\label{sec:Leaf}
In this appendix, we show how the computations in section \ref{sec:sec} and \ref{sec:subharm} work in the limiting case of Problem \ref{ProblemDisc} when $\epsilon=0$, or Problem \ref{ProblemHomoDisc}.  The manifold $V$ is of arbitrary dimension.
 
 To carry out the computation, it's necessary to make the following assumptions:
 \begin{enumerate}
 	\item $\Phi\in C^4( \MR\times \MV),$%\cap C^2(\overline \MR\times \MV)$
 	\item $\omega_0+\sqrt{-1}\ddbar \Phi(\tau, \ast)>0$.
 \end{enumerate}
However these assumptions may not be satisfied in general as shown by \cite{LempertLizVivas}\cite{LempertDarvas}\cite{RossNystrom} and \cite{HuIMRN}. So the computations are formal. But the computations in this section provide a guideline for Section \ref{sec:sec} and \ref{sec:subharm}.

According to \cite{DonaldsonHolomorphicDiscs}, under  assumptions 1 and 2, kernels of $\Omega_0+\sqrt{-1}\ddbar\Phi$ form a foliation in $\MR\times V$, and at a point $(\tau_0, p_0)\in \MR\times \MV$, the leaf is locally the graph of a holomorphic map, $f: \MU\rightarrow \MV$, where $\MU$ is an open set in $\MR$, $ \MU\ni \tau_0$ and $f(\tau_0)=p_0.$

In the case that $\MR$ is a disc, each leaf is also a disc, and $f$ can be defined on $\MR$. However, when $\MR$ is not simply connected, a leaf is a covering of $\MR$, of finite or infinite index. More information on the foliation structure can be found in \cite{DonaldsonHolomorphicDiscs}
and \cite{ChenTian}.

In the following we will do computations locally on an open set of a leaf: \[L\triangleq \{(\tau,f(\tau))|\tau\in U\}.\]
Denote the projection $\MR\times \MV\rightarrow \MR$ by $\pi_\MR$, then $\pi^\ast_\MR \tau\triangleq \X$ is a complex coordinate on $L$.
Actually $\X=\pi_\MR$, if we consider $\pi_\MR$ as a map to $\EC\ (\EC\supset\MR)$. Because the tangent vector of the leaf lies in the kernel of $\Omega_0+\sqrt{-1}\ddbar\Phi$, we have
\[\partial_\X=\partial_\tau-\Phi_{\tau\betabar}g^{\alpha\betabar}\partial_\alpha.\]
Consider $A=\Phi_{\alpha\betabar}dz^\alpha\otimes  dz ^{\betabar}$ and 
$B=\Phi_{\alpha\beta}dz^\alpha\otimes dz^{\beta}$ as sections of  bundles $i^\ast\pi_V^\ast(T^{\ast 1,0}(V)\otimes T^{\ast 0,1}(V))$ and $i^\ast\pi_V^\ast(T^{\ast 1,0}(V)\otimes T^{\ast 1,0}(V))$ over $L$. Here $i$ is the inclusion map $L\hookrightarrow \MR\times V$. 

With these notations, equations (\ref{equationA}) (\ref{equationBsum}) and (\ref{equationBbarsum}) become
\begin{align}
	A_{\alpha\betabar, \X\Xbar}=A_{\alpha\thetabar,\X}g^{\mu\thetabar}A_{\mu\betabar,\Xbar}+B_{\alpha\theta,\Xbar}g^{\theta\gammabar}\overline{B_{\gamma\beta, \Xbar}}+U_{\alpha\betabar};\\
	B_{\alpha\beta, \X\Xbar}=A_{\alpha\thetabar,\X}g^{\mu\thetabar}B_{\mu\beta,\Xbar}+B_{\alpha\theta,\Xbar}g^{\theta\gammabar}{A_{\beta\gammabar, \X}}+V_{\alpha\beta};\\
\overline{B_{\alpha\beta}}_{,\X\Xbar}=A_{\theta\alphabar,\Xbar}g^{\theta\mubar}\overline{B_{\mu\beta,\Xbar}}+\overline{B_{\alpha\theta}}_{,\X}g^{\gamma\thetabar}{A_{\gamma\betabar, \Xbar}}+W_{\alphabar\betabar}.
\end{align}
Terms $F, H$ in (\ref{equationA}) (\ref{equationBsum}) and (\ref{equationBbarsum}) vanish because $\epsilon=0$. See the expression of $F, H$ (\ref{expressF}) and (\ref{expressH}).
It's convenient to consider $A, B, U, V, W$ as matrices, then we can use matrix multiplication to simplify the equations to:
\begin{align}
	A_{,\X\Xbar}=A_\cX G^{-1} A_{\cXbar}+ B_{,\Xbar}\overline{G^{-1}B_{\cXbar}}+U,\label{eq:appB:A}\\
	B_{,\X\Xbar}=A_\cX G^{-1} B_{\cXbar}+B_{\cXbar}\overline{G^{-1}A_{\cXbar}}+V,\label{eq:appB:B}\\
	\overline{B}_{,\X\Xbar}=\overline{B}_\cX G^{-1}A_{\cXbar}+\overline{A}_\cXbar
	\overline{G}^{-1}\overline{B}_\cX+W.\label{eq:appB:Bbar}
	\end{align}
In the equations above, $G=(g_{\alpha\betabar})$ and $G^{-1}=(g^{\beta\alphabar})$, where $\alpha$ is the row index and $\beta$ is the column index. With matrix multiplications, the quantity $Q_B$, introduced in (\ref{expressQB}), is 
\begin{align}
	Q_B=\tr\left({B\overline{G^{-1}B}G^{-1}}\right).  \label{expressQBmatrix}
\end{align}

In the case that curvatures of $(\MV, \omega_0)$ are $0$, terms $U, V, W$ all vanish and computations will be largely simplified. We choose a coordinate at the point of computation, so that $g_{\alpha\betabar}=\delta_{\alpha\betabar}$ and the Christoffel symbols are zero.  Then apply $\X $ and $\Xbar$ to (\ref{expressQBmatrix}). By Leibniz rule and equations (\ref{eq:appB:A})-(\ref{eq:appB:Bbar}), we get
\begin{align}
	&(Q_B)_{\X\Xbar}\\
	=\tr&\left(B_\Xbar\Abar_\X\Bbar+A_\X B_\Xbar\Bbar
	                    -B_\X\Abar_\Xbar \Bbar+B_\X\Bbar_\Xbar-B_\X \Bbar A_\Xbar  -B_\Xbar\Abar_\X\Bbar\right.
	                                \label{2401}\\
	                 &   +B\Abar_\Xbar\Abar_\X\Bbar
	                       -B\Abar_\Xbar\Abar_\X\Bbar-B\Bbar_\X B_\Xbar \Bbar+B\Abar_\X\Abar_\Xbar\Bbar-B\Abar_\X \Bbar_\Xbar+B\Abar_\X\Bbar A_\Xbar
	                              \label{2402}\\
	                  &+B_\Xbar \Bbar_\X-B\Abar_\Xbar\Bbar_\X +B\Bbar_\X A_\Xbar
	                  +B\Abar_\Xbar \Bbar_\X-B\Bbar_\X A_\Xbar-B_\Xbar\Bbar A_\X
	                  \label{2403}\\
	                  &\left.+B\Abar_\Xbar \Bbar A_\X -B\Bbar_\Xbar A_\X+B\Bbar A_\Xbar A_\X-B\Bbar A_\X A_{\Xbar} +B\Bbar A_\X A_\Xbar-B\overline{B} B_\Xbar\Bbar_\X\right)
	             \label{2404}.
\end{align}
There are 12 terms canceling each other:
\begin{align}
(\ref{2401})_2+(\ref{2401})_6=0, 
\ (\ref{2401})_1+(\ref{2403})_6=0,
\ (\ref{2403})_2+(\ref{2403})_4=0,  \label{firstrow}
\end{align} 
\begin{align}
\ (\ref{2403})_3+(\ref{2403})_5=0, 
\ (\ref{2402})_1+(\ref{2402})_2=0, 
\ (\ref{2404})_4+(\ref{2404})_5=0.   \label{secondrow}
\end{align}
Here we are using a notation mentioned at (\ref{termnotation}), term $(\ast.\ast)_k$ stands for the $k-$th term in $(\ast.\ast)$, for example $(\ref{2401})_6=-\tr(B_\Xbar\Abar_\X\Bbar)$.
For the first two equalities of (\ref{firstrow}), we used that transposing a matrix does not change its trace and for matrices $A_1, A_2, ... A_k$, 
\begin{align}
	\tr{(A_1 A_2\cdot\cdot\cdot A_k )}=\tr{(A_k A_1 A_2\cdot\cdot\cdot A_{k-1})}.
\end{align}
So 
\begin{align}
	(\ref{2401})_2=\tr(A_XB_{\Xbar}\overline{B})=\tr[(A_XB_{\Xbar}\overline{B})^T]=\tr(\Bbar B_{\Xbar} \Abar_X)=\tr(B_{\Xbar} \Abar_X\Bbar )=-(\ref{2401})_6.
\end{align}
The second equality of (\ref{firstrow}) can be derived similarly, and the last equality of (\ref{firstrow}) and equalities of (\ref{secondrow}) are trivial.
Then we reorganize the remaining terms together and get
\begin{align}\label{LaplaceQBideal}
	(Q_B)_{\X\Xbar}
	=\tr\left(B_\Xbar\Bbar_\X-2\Bbar B \Bbar_X B_\Xbar\right)+ \tr\left[\left(B_\X-A_\X B-B \Abar_\X\right)\left(\Bbar_\Xbar-\Abar_\Xbar \Bbar-\Bbar A_\Xbar\right)\right].
\end{align}
So $(Q_B)_{\X\Xbar}\geq 0$, providing the maximal eigenvalue of $B\Bbar<\frac{1}{2}$ (or the maximal eigenvalue of $B\overline{G^{-1}B}G^{-1}<\frac{1}{2}$, without the assumption that $(g_{\alpha\betabar})$ is diagonal).   This result will be improved  to $B\Bbar<1$ in an upcoming paper.

When the curvatures of $(\MV, \omega_0)$ are not zero, terms $U, V, W$ are involved and (\ref{LaplaceQBideal}) becomes
\begin{align}
	(Q_B)_{\X\Xbar}
	=&\tr\left(B_\Xbar\Bbar_\X-2\Bbar B \Bbar_X B_\Xbar\right)+ \tr\left[\left(B_\X-A_\X B-B \Abar_\X\right)\left(\Bbar_\Xbar-\Abar_\Xbar \Bbar-\Bbar A_\Xbar\right)\right]
	\label{negativeterms0624-1}\\
	&+\tr\left(V \overline B+BW-B\overline{U}\Bbar-B\Bbar U\right). \label{negativeterms0624-2}
\end{align}
To control (\ref{negativeterms0624-2}), $Q_A, Q_G$ (defined in (\ref{expressQA}) (\ref{expressQG})) need to be introduced. With the notation of this appendix, 
\begin{align}
	Q_A=\tr(AG^{-1}A G^{-1}).
\end{align}In addition, $S$, which is  a section of $i^\ast \pi^\ast_\MV T^{\ast 1,0}(\MV)$,  also gets involved when computing $(Q_G)_{\X\Xbar}$. Here, with $\epsilon=0$, equation (\ref{expressS}) becomes
 \begin{align}
 	S_\theta=\Phi_{\theta,\X\Xbar}={R_{\zetabar\theta\mu}}^\psi \Phi_\psi g^{\mu\betabar} \Phi_{\tau\betabar}g^{\alpha\zetabar}\Phi_{\alpha\taubar}.
 \end{align}

In the following, we want to derive an estimate for $Q_{X\Xbar}$, with $Q=Q_A+Q_B+Q_G$.
First, we transform (\ref{negativeterms0624-1}) (\ref{negativeterms0624-2}) into an inequality. For a constant $C$ depending on dim($\MV$), providing $Q<\frac{1}{2}$, we have
\begin{align}\label{EstimateQBappendixB}
	(Q_B)_{\X\Xbar}\geq (1-2Q)\tr \left(\Bbar_\X B_{\Xbar}\right)-C(|V|+|U|+|W|)\sqrt{Q}.
\end{align}
As introduced in Section \ref{sec:notation},  $|\ast|$ is the norm of a tensor, with respect to the metrics induced by $\omega_0$. Similarly, for $Q_A$ and $Q_G$, we have that
\begin{align}\label{EstimateQAappendixB}
	(Q_A)_{\X\Xbar}\geq (2-C \sqrt{Q}) \tr \left(A_\X A_{\Xbar}\right)-C|U|\sqrt{Q}-C\sqrt{Q}\ \tr \left(\Bbar_\X B_{\Xbar}\right),
\end{align}
and 
\begin{align}\label{EstimateQGappendixB}
	(Q_G)_{\X\Xbar}\geq \frac{1}{2}\left(\Phi_{\theta, \X}\Phi_{\thetabar, \Xbar}
	                                                             +  \Phi_{\thetabar, \X}\Phi_{\theta, \Xbar}
	                                                               \right)-C(|S|+|U|)\sqrt{Q}.
\end{align}
To control $|S|, |U|, |V|, |W|$, we define
\begin{align}
	\MP=B_{\alpha\beta, \Xbar}\overline{B_{\theta\gamma, \Xbar}}g^{\alpha\thetabar}g^{\beta\gammabar}
			+A_{\theta\betabar, \X}{A_{\alpha\gammabar, \Xbar}}g^{\alpha\betabar}g^{\theta\gammabar}
			+\Phi_{\theta,\X}\Phi_{\gammabar,\Xbar}g^{\theta\gammabar}
			+\Phi_{\theta,\Xbar}\Phi_{\gammabar, \X}g^{\theta\gammabar}.
\end{align}
At the point of computation, where $g_{\ijbar}=\delta_{\ijbar}$, 
\begin{align}
	\MP=\tr(B_{\Xbar}\overline{B}_{ \X}
	+A_{ \X}{A_{\Xbar}})
	+\Phi_{\theta\X}\Phi_{\thetabar\Xbar}
	+\Phi_{\theta\Xbar}\Phi_{\thetabar \X}.
\end{align}
 $\MP$ defined here is roughly a combination of $P$ and $T$, defined in (\ref{expressP}) and (\ref{expressT}). It's easy to see that 
\begin{align}
	\MP\geq P+  (1-C\sqrt{Q}) T.
\end{align}
So, when $Q$ is small enough, we can simplify  (\ref{sumcontrol}) and get the following estimates for $|U|, |V|, |W|, |S|$
\begin{align}
	|U|+|V|+|W|+|S|\leq C_R  \MP.
\end{align} Here $C_R$ is a constant depending on curvatures of $\MV$ and their covariant derivatives.
Then combining  (\ref{EstimateQBappendixB}) (\ref{EstimateQAappendixB}) and (\ref{EstimateQGappendixB}) together, we find, when $Q$ is smaller than a dimensional constant, 
\[Q_{\X\Xbar}\geq \frac{1}{2}\MP-C_R \sqrt{Q}\ \MP.\]
When $Q$ is smaller than a curvature related constant, we have $Q_{\X\Xbar}\geq 0$. This leads to an upper bound for $Q$, and so  an upper bound for $|D^2\Phi(\tau, \ast)|_{\omega_0}$, $\tau\in\MR$.
%%%%%%%%%%%%%%%%%%%%%%%%%%%%%%%%%%%%%
%%%%%%%%%%%%%%%%%%%%%%%%%%%%%%%%%%%%
\section*{Acknowledgement}
This work is supported by the National Natural Science Foundation of China (No. 12288201) and the Project of Stable Support for Youth Team in Basic Research Field, CAS, (No. YSBR-001). The author would like to thank  Xiuxiong Chen, Laszlo Lempert and Bing Wang for helpful discussions and their suggestions on improving the paper.
% He also wants to thank the anonymous referee for his/her many insightful and thoughtful suggestions on improving the paper.
%Jianchun Chu, Guohuan Qiu

%%%%%%%%%%%%%%%%%%%%%%%%%%%%%%%%%%%%%%%%%%%%%%%

{\flushleft Jingchen Hu}\\
Hua Loo-Keng Center for Mathematical Sciences, Chinese Academy of Sciences\\
Email:
JINGCHENHOO@GMAIL.COM

\end{document}